\documentclass[journal]{IEEEtran}
\usepackage[english]{babel}
\usepackage[latin1]{inputenc}
\usepackage{enumerate}
\usepackage{color}
\usepackage[T1]{fontenc}
\usepackage{subfigure}
\usepackage{dsfont}
\usepackage[T1]{fontenc}
\usepackage{amsmath}
\usepackage{amsthm}
\usepackage{amstext}
\usepackage{amssymb}
\usepackage{mathrsfs}
\usepackage{mathtools}
\usepackage{tikz}

\def\R{\mathbb{R}}

\def\B{{\mathcal B}}
\def\C{{\mathcal C}}

\def\G{{\mathcal G}}
\def\P{{\mathcal P}}
\def\Q{{\mathcal Q}}
\def\S{{\mathcal S}}

\def\sM{{\mathsf M}}

\def\sU{{\mathsf U}}
\def\sX{{\mathsf X}}

\def\heta{\hat{\eta}}

\fussy

\newtheorem{definition}{Definition}
\newtheorem{theorem}{Theorem}
\newtheorem{assumption}{Assumption}

\newtheorem{proposition}{Proposition}
\newtheorem{lemma}{Lemma}
\theoremstyle{remark}
\newtheorem{remark}{Remark}
\IEEEoverridecommandlockouts

\allowdisplaybreaks

\begin{document}
\sloppy
\title{On Optimal Zero-Delay Coding of\\ Vector Markov Sources
\thanks{The authors are with the Department of Mathematics and
    Statistics, Queen's University, Kingston, Ontario, Canada, K7L
    3N6.  Email: (yuksel,linder)@mast.queensu.ca.}
   \thanks{This research was
    partially supported by the Natural Sciences and Engineering
    Research Council of Canada (NSERC).}
    \thanks{The material of this paper was presented in part at the 51st IEEE
      Conference on Decision and Control (Maui, Hawaii, Dec.\ 10-13, 2012) and at the 2013 Workshop on Sequential and
Adaptive Information Theory (Montreal, Quebec, Nov.\ 7-9, 2013). }
\thanks{\copyright 2014 IEEE. Personal use
of this material is permitted.  However, permission to use this material for
any other purposes must be obtained from the IEEE by sending a request to
pubs-permissions@ieee.org. }
\thanks{DOI 10.1109/TIT.2014.2346780}
 }
\author{Tam\'as Linder and Serdar Y\"uksel }
\maketitle
\begin{abstract}
  Optimal zero-delay coding (quantization) of a vector-valued Markov source
  driven by a noise process is considered. Using a stochastic control
  problem formulation, the existence and structure of optimal quantization
  policies are studied.  For a finite-horizon problem with bounded per-stage
  distortion measure, the existence of an optimal zero-delay quantization policy is
  shown provided that the quantizers allowed are ones with convex codecells. The
  bounded distortion assumption is relaxed to cover cases that include the linear
  quadratic Gaussian problem. For the infinite horizon problem and a stationary
  Markov source the optimality of deterministic Markov coding policies is
  shown. The existence of optimal stationary Markov quantization policies is
  also shown provided randomization that is shared by the encoder and the
  decoder is allowed.
\end{abstract}
%\end{frontmatter}
\begin{keywords}
Real-time source coding, Markov source, quantization, stochastic control, Markov
decision processes.
\end{keywords}

%\pagestyle{myheadings}
%\thispagestyle{plain}
%\markboth{P. DUGGAN AND V. A. U. THORS}{SIAM MACRO EXAMPLES}
\section{Introduction}

\subsection{Zero-delay coding}

We consider a zero-delay (sequential) encoding problem where a sensor encodes an
observed information source without delay. It is assumed that the information
source $\{x_t\}_{t\ge 0}$ is a time-homogenous $\R^d$-valued discrete-time
Markov process.  The  initial distribution $\pi_0$ (i.e., the distribution of
$x_0$) and  the transition kernel $ P(dx_{t+1}|x_t)$
uniquely determine the distribution of $\{x_t\}_{t\ge 0}$.  In
Assumption~\ref{AssumptionA} below we will make explicit assumptions about the
transition kernel.

The encoder encodes (quantizes) the source samples and transmits the encoded
versions to a receiver over a discrete noiseless channel with input and output
alphabet $\sM\coloneqq\{1,2,\ldots,M\}$, where $M$ is a positive integer.
Formally, the encoder is specified by a \emph{quantization policy} $\Pi$, which
is a sequence of Borel measurable functions $\{\eta_t\}_{t\ge 0}$ with $\eta_t:
\sM^t \times (\R^d)^{t+1} \to \sM$.  At time $t$, the encoder transmits the
$\sM$-valued message
\[
  q_t=\eta_t(I_t)
\]
with $I_0=x_0$, $I_t=( q_{[0,t-1]} ,x_{[0,t]})$ for $t \geq 1$, where we
have used the notation $q_{[0,t-1]}=(q_0,\ldots,q_{t-1})$ and $x_{[0,t]} =
(x_0,x_1,\ldots,x_t)$.  The collection of all such zero-delay policies is called the
set of admissible quantization policies and is denoted by $\Pi_A$.

Observe that for fixed $q_{[0,t-1]}$ and $x_{[0,t-1]}$, as a
function of $x_t$, the encoder
$\eta_t(q_{[0,t-1]},x_{[0,t-1]},\,\cdot\,)$ is a \emph{quantizer}, i.e., a Borel
measurable mapping of $\R^d$ into the finite set~$\sM$. Thus any quantization
policy at each time $t\ge0$ selects a quantizer $Q_t:\R^d\to \sM$  based on past
information $(q_{[0,t-1]},x_{[0,t-1]})$,  and then ``quantizes'' $x_t$ as
$q_t=Q_t(x_t)$.

Upon receiving $q_t$, the receiver
generates its reconstruction $u_t$, also without delay. A zero-delay 
receiver policy is a sequence of measurable functions
$\gamma=\{\gamma_t\}_{t\ge 0}$ of type  $\gamma_t : \sM^{t+1} \to
\sU$, where  $\sU$ denotes the reconstruction alphabet  (usually a Borel subset of
$\R^d$).  Thus
\[
u_t=\gamma_t(q_{[0,t]}), \qquad t\ge 0.
\]

For the  finite horizon setting the goal is to  minimize the average cumulative
cost (distortion) 
\begin{equation}\label{Cost1}
J_{\pi_0}(\Pi,\gamma,T) \coloneqq
E^{\Pi,\gamma}_{\pi_0}\biggl[\frac{1}{T} \sum_{t=0}^{T-1} c_0(x_t,u_t)\biggr],
\end{equation}
for some $T \ge 1$, where $c_0: \R^d \times \sU \to \mathbb{R}$ is a
\emph{nonnegative} Borel measurable cost (distortion) function.  Here
$E^{\Pi,\gamma}_{\pi_0}$ denotes expectation under initial distribution $\pi_0$
for $x_0$; the superscript signifies that the argument is a function of
$\{x_t\}$ which depends on the quantization policy $\Pi$ and receiver policy
$\gamma$. (Later we will use the notation $E^{\Pi}_{\pi_0}$ for expectations
where the argument is a function of $\{x_t\}$ that depends only on $\Pi$, and
the notation $E_{\pi_0}$ when the argument has no dependence on either $\Pi$ or
$\gamma$.) We assume that the encoder and decoder know the initial distribution
$\pi_0$.

We also consider the infinite-horizon average cost problem where the objective
is to minimize
\begin{eqnarray}
J_{\pi_0}(\Pi,\gamma)& \coloneqq & \limsup_{T \to \infty}
J_{\pi_0}(\Pi,\gamma,T)  \nonumber \\
& = &\limsup_{T \to \infty} 
E^{\Pi,\gamma}_{\pi_0}\biggl[\frac{1}{T}   \sum_{t=0}^{T-1} c_0(x_t,u_t)\biggr].\label{infiniteCost}
\end{eqnarray}

Our main assumption on the Markov source  $\{x_t\}$ is the following.

\begin{assumption}\label{AssumptionA}  
 The evolution of $\{x_t\}$ is given by
\begin{equation}
\label{sourceChannelModel11}
x_{t+1} = f(x_t,w_t), \quad t =0, 1, 2,\ldots,
\end{equation}
where $f: \R^d\times \R^d \to \R^d$ is a Borel function and $\{w_t\}$ is an
independent and identically distributed (i.i.d$.$) vector noise sequence which
is independent of $x_0$. It is assumed that for each fixed $x\in \R^d$, the
distribution of $f(x,w_t)$ admits the (conditional)
density function $\phi(\,\cdot\,|x)$ (with respect to the $d$-dimensional
Lebesgue measure)  which is  positive everywhere. Furthermore,
$\phi(\,\cdot\,|x)$ is bounded and Lipschitz uniformly in~$x$.
\end{assumption}

The above model includes the linear systems with Gaussian noise. Further
conditions on $f$ and the cost $c_0$, and the reconstruction alphabet $\sU$ will be given
in Sections \ref{sec_finite} and \ref{LQGCaseSection} for the finite-horizon
problem (these include the case of a linear system and quadratic cost) and in
Section~\ref{sec_infinite} for the infinite-horizon problem.

Before proceeding  further with formulating the results, we provide an overview
of structural results for finite-horizon optimal zero-delay coding problems as
well as
a more general literature review.

\subsection{Revisiting structural results for finite-horizon problems}
\label{sec_structure}

Structural results for the finite horizon control problem described in the
previous section have been developed in a number of important papers. Among
these the classic works by Witsenhausen \cite{Witsenhausen} and  Walrand
and Varaiya \cite{WalrandVaraiya}, using two different approaches, are of
particular relevance.  Teneketzis \cite{Teneketzis} extended these
approaches to the more general setting of non-feedback communication and
\cite{YukITArXivMultiTerminal} extended these results to more general state
spaces (including $\R^d$). The
following two theorems summarize, somewhat informally, these two important
structural results.

\begin{theorem}[Witsenhausen \cite{Witsenhausen}] \label{witsenhausenTheorem} For
  the finite horizon problem, any  zero-delay quantization policy $\Pi=\{\eta_t\}$ can
  be replaced, without any loss in performance, by a policy
  $\hat{\Pi}=\{\heta_t\}$ which only uses
  $q_{[0,t-1]}$ and $x_t$ to generate $q_t$, i.e., such that
  $q_t=\heta_t(q_{[0,t-1]},x_t)$ for all $t=1,\ldots,T-1$.
\end{theorem}

For a complete and separable (Polish) metric space $\sX$ and its Borel sets
$\B(\sX)$, let ${\cal P}(\sX)$ denote the space of probability measures on $(\sX,\B(\sX))$,
endowed with the topology of weak convergence (weak topology). This topology is
metrizable with the Prokhorov metric making $\P(\sX)$ itself a Polish space.
Given a quantization policy $\Pi$, for all $t\ge 1$ let $\pi_t \in {\cal
  P}(\R^d)$ be the regular conditional probability defined by
 \[
\pi_t(A)\coloneqq P(x_t\in A | q_{[0,t-1]})
\]
for any Borel set $A\subset \R^d$.

The following result is due to Walrand and Varaiya \cite{WalrandVaraiya} who
considered sources taking values in a finite set. For the more general case of $\R^d$-valued sources the
result appeared in \cite{YukITArXivMultiTerminal}.

\begin{theorem}\label{WalrandVaraiyaTheorem} For the  finite horizon
  problem, any zero-delay quantization policy can be replaced, without any loss
  in performance, by a policy which at any time $t= 1, \ldots,T-1$ only uses the
  conditional probability measure $\pi_t=P(dx_t|q_{[0,t-1]})$ and the state
  $x_t$ to generate $q_t$. In other words, at time $t$ such a policy $\heta_t$
  uses $\pi_t$ to select a quantizer $Q_t=\heta(\pi_t)$ (where $Q_t: \R^d \to
  \sM$),  and then $q_t$ is generated as $q_t=Q_t(x_t)$.
\end{theorem}

A policy of the type suggested by Theorem~\ref{WalrandVaraiyaTheorem} (a
so-called Walrand-Varaiya-type policy)  is called \emph{stationary} if $\heta_t=
\heta$ for all $t$, where $\heta$ is a fixed policy mapping elements of
$\P(\R^d)$  to the set of $M$-level quantizers. Stationary 
policies will play an important role in Section~\ref{sec_infinite}. 

As discussed in \cite{YukITArXivMultiTerminal}, the main difference between the
two structural results above is the following: In the setup of
Theorem~\ref{witsenhausenTheorem}, the encoder's memory space is not fixed and
keeps expanding as the encoding block length $T$ increases. In the setup of
Theorem~\ref{WalrandVaraiyaTheorem}, the memory space of an optimal encoder is
fixed (note that $\pi_t$ can be computed from $\pi_{t-1}$, $Q_{t-1}$, and
$q_{t-1}$; see equation \eqref{filtre}).  Of course, in general the space of
probability measures is a very large one. However, it may be the case that
different quantization outputs lead to the same conditional probabilities
$\pi_t$, leading to a reduction in the required memory. More importantly, the
setup of Theorem~\ref{WalrandVaraiyaTheorem} allows one to apply the powerful
theory of Markov Decision Processes on fixed state and action spaces, thus
greatly facilitating the analysis.

In this paper, we show that under quite general assumptions on the Markov
process, the cost function, and the admissible quantization policies there
always exists a Walrand-Varaiya-type policy  that
minimizes the finite horizon cost \eqref{Cost1}. For the infinite horizon
problem \eqref{infiniteCost}, we show that there exists an optimal
Walrand-Varaiya-type policy if the source is stationary.  We also show that in
general an optimal (possibly randomized) stationary quantization policy exists
in the set of Walrand-Varaiya-type policies.

The rest of the paper is organized as follows. The next section gives a brief
review of the literature.  Section~\ref{SectionSpaceQuant} contains background
material on quantizers and the construction of a controlled Markov chain for our
problem. Section~\ref{sec_finite} establishes the existence of optimal policies
for the finite horizon case for bounded cost functions. Section
\ref{LQGCaseSection} considers the quadratic
costs under conditions that cover linear systems. Section~\ref{sec_infinite} considers the more involved infinite horizon
case. Section~\ref{sec_conclusions} contains concluding discussions. Most of the
proofs are relegated to the Appendix.

\subsection{Literature review and contributions}
The existence of optimal quantizers for a one-stage ($T=1$) cost problem has
been investigated in \cite{AbayaWise}, \cite{Pollard}, and
\cite{YukLinSIAM2010},  among other works.

An important inspiration for our work is Borkar \emph{et al$.$}
\cite{BorkarMitterTatikonda} which studied the optimal zero-delay quantization of
Markov sources.  For the infinite horizon setting, this paper provided a
stochastic control formulation of the optimal quantization problem with a
Lagrangian cost that combined squared distortion and instantaneous entropy, and
gave an elegant proof for the existence of optimal policies.

It should be noted that \cite{BorkarMitterTatikonda} restricted the admissible
quantizers $Q_t$ at each time stage $t$ to so-called nearest neighbor quantizers
whose reconstruction values were also suboptimally constrained to lie within a
fixed compact set. Furthermore, some fairly restrictive
conditions were placed on the dynamics of the system. These include
requirements on the system dynamics that rule out additive noise models with unbounded
support such as the Gaussian noise (see p.~138 in \cite{BorkarMitterTatikonda}),
and a uniform Lipschitz condition on the cost functions (see the condition on
$\hat{f}$ on p. 140 in \cite{BorkarMitterTatikonda}). These conditions made it
possible to apply the discounted cost approach (see, e.g., \cite{survey}) to average cost
optimization problems.

Furthermore, the encoder-decoder structure in \cite{BorkarMitterTatikonda} has
been specified {\it a priori}, whereas in this paper, we only relax global
optimality when we restrict the quantizers to have convex codecells (to be
defined later), which is a more general condition than assuming the nearest
neighbor encoding rule. On the other hand, we are unable to claim the optimality
of deterministic stationary quantization policies for the infinite-horizon
problem, whereas \cite{BorkarMitterTatikonda} establishes optimality of
such policies. However, as mentioned, the conditions on the cost
functions, systems dynamics, and the uniform continuity condition over all
quantizers are not required in our setting.

To our knowledge, the existence of optimal quantizers for a finite horizon setting has not been considered in the
literature for the setup considered in this paper.

Other relevant work include \cite{BorkarMitterSahaiTatikonda} which considered
optimization over probability measures for causal and non-causal settings, and
\cite{Teneketzis}, \cite{Mahajan09}, \cite{MahajanTeneketzisJSAC} and
\cite{YukITArXivMultiTerminal} which considered zero-delay coding of Markov
sources in various setups. Structural theorems for zero-delay variable-rate
coding of discrete Markov sources were studied in \cite{KaMe12}.  Recently
\cite{AsWe13} considered the average cost optimality equation for coding of
discrete i.i.d.\ sources with limited lookahead and \cite{JaGo13} studied
real-time joint source-channel coding of a discrete Markov source over a discrete
memoryless channel with  feedback.

A different  model for sequential source coding, called 
causal source coding, is  studied in, e.g., \cite{NeGi82,WeMe05,LiZa06}. In causal
coding,  the reconstruction depends causally
on the source symbols, but in the information transmission process large delays
are permitted, which makes this model  less stringent (and one might argue
less practical) than zero or limited-delay source coding. 

For systems with control, structural results have also been investigated in the
literature. In particular, for linear systems with quadratic cost criteria
(known as LQG optimal control problems), it has been shown that the effect of
the control policies can be decoupled from the estimation error without any
loss.  Under optimal control policies, \cite{YukselAllerton2012} has shown the
equivalence with the control-free setting considered in this paper (see also
\cite{BaSkJo11} and \cite{NaFaZaEv07} for related results in different
structural forms, where in contrast with
\cite{YukselAllerton2012} the encoders have memory).  We also note that the design
results developed here can be used to establish the existence of optimal
quantization and control policies for LQG systems \cite{YukselAllerton2012}.

{\bf Contributions:} In view of the literature review, the main contributions of the
paper can be summarized as follows.

\begin{itemize}
\item [(i)] We establish a useful topology on the
set of quantizers, building on \cite{YukLinSIAM2010}, among other works, and
show the existence of optimal coding policies for finite horizon optimization
problems, under the assumption that the quantizers used have convex
codecells. Notably, the set of sources considered includes LQG systems, i.e.,
linear systems driven by Gaussian noise under the quadratic cost criterion. The
analysis requires the development of a series of technical results which
facilitate establishing measurable selection criteria, reminiscent of those in
\cite{HernandezLermaMCP}.

\item[(ii)] We establish, for the first time to our
knowledge, the optimality of Markov (i.e., Walrand-Varaiya type) coding policies for {\it infinite-horizon}
sequential quantization problems, using a new approach. The prior work reviewed above strictly build on dynamic programming (which is only suitable for finite-horizon problems) or does not consider the question of global optimality of Markov policies.

\item[(iii)] We show the existence of optimal
stationary, possibly randomized, policies which are globally optimal, for a
large class of sources including LQG systems. As detailed above, the
assumptions are weaker than those that have appeared in prior work.
\end{itemize}

\section{ Quantizer actions and controlled Markov process construction}\label{SectionSpaceQuant}
In this section, we formally define the space of quantizers considered in the
paper building  on the construction in \cite{YukLinSIAM2010}.  Recall the notation
$\sM=\{1,\ldots,M\}$.

\begin{definition} \label{quant_def}  An $M$-cell quantizer $Q$ on $\R^d$ is a (Borel) measurable
  mapping $Q:\R^d\to \sM$. We let $\mathcal{Q}$ denote the
  collection of all $M$-cell quantizers on $\R^d$.
\end{definition}

Note that each $Q\in \mathcal{Q}$ is uniquely characterized by
its \emph{quantization cells} (or bins)  $B_i=Q^{-1}(i) =  \{x: Q(x)=i\} $,
$i=1,\ldots, M$ which form a measurable
partition of~$\R^d$.

\smallskip

\pagebreak[2]

\begin{remark} 
\mbox{}
\begin{itemize}
\item[(i)] We allow for the possibility that
  some of the cells of the  quantizer  are empty.
\item[(ii)] In source coding theory (see, e.g., \cite{GrNe98}), a quantizer is a mapping $Q: \, \R^d
  \to \R^d$ with a finite range. In this definition, $Q$ is specified by a
  partition $\{B_1,\ldots,B_M\}$ of $\R^d$ and reconstruction values $\{
  c_1,\ldots,c_M\}\subset \R^d$ through the mapping rule $Q(x) = c_i$ if
  $x \in B_i$. In our definition, we do not include the reconstruction values.
\end{itemize}
\end{remark}

In view of Theorem~\ref{WalrandVaraiyaTheorem}, any admissible quantization
policy can be replaced by a Walrand-Varaiya-type
policy. The class of all such policies is denoted by
$\Pi_W$ and is formally defined as follows.

\begin{definition} \label{WVdef} An (admissible) quantization policy $\Pi=\{\eta_t\}$
  belongs to $\Pi_W$ if there exist a sequence of mappings $\{\heta_t\}$ of
  the type $\heta_t: \P(\R^d) \to \Q$ such  that for $Q_t=\heta_t(\pi_t)$
  we have $q_t=Q_t(x_t)=\eta_t(I_t)$.
\end{definition}

Suppose we use a quantizer policy $\Pi=\{\heta_t\}$ in $\Pi_W$.  Let
$P(dx_{t+1}|x_t)$ denote the transition kernel of the process $\{x_t\}$ determined
by the system dynamics \eqref{sourceChannelModel11} and note that $P(q_t| \pi_t,
x_t)$ is determined by the quantizer policy as $P(q_t | \pi_t,x_t) =
1_{\{Q_t(x_t)=q_t\}}$, where $Q_t=\heta_t(\pi_t)$ and $1_A$ denotes the
indicator of event $A$. Then standard properties of
conditional probability can be used to obtain the following filtering equation
for the evolution of $\pi_t$:
\begin{eqnarray}
\pi_{t+1}(dx_{t+1}) \!\!\!\! & = & \!\!\!  \frac{P(dx_{t+1}, q_t| q_{[0,t-1]})}{ P(q_t|
  q_{[0,t-1]})   }{} \nonumber \\*
&=&  \!\!\! \frac{\int_{\R^d} \pi_t(dx_t) P(q_t | \pi_t, x_t)
  P(dx_{t+1}|x_t)}{\int_{\R^d} \int_{\R^d} \pi_t(dx_t) P(q_t
  | \pi_t, x_t) P(dx_{t+1}|x_t)} \nonumber \\
&=&  \!\!\!\!\!  \frac{1}{\pi_t(Q^{-1}(q_t))} \int_{Q^{-1}(q_t)} \!\!\!\!\!\!\!\! P(dx_{t+1}|
x_t)\pi_t(dx_t).  \label{filtre}
\end{eqnarray}
Hence $\pi_{t+1}$ is determined by $\pi_t$, $Q_t$, and $q_t$, which
implies that $\pi_{t+1}$ is conditionally independent of
$(\pi_{[0,t-1]},Q_{[0,t-1]})$ given $\pi_t$ and $Q_t$. Thus $\{\pi_t\}$ can be
viewed as $\P(\R^d)$-valued controlled Markov process
\cite{HernandezLermaMCP,HernandezLermaMCP1} with $\Q$-valued control $\{Q_t\}$
and average cost up to time $T-1$ given by
\[
 E^{\Pi}_{\pi_0}\biggl[ \frac{1}{T}\sum_{t=0}^{T-1}
  c(\pi_t,Q_t)\biggr]= \inf_{\gamma} J_{\pi_0}(\Pi,\gamma,T),
\]
where
\begin{equation}
\label{eq_cdef}
c(\pi_t,Q_t): = \sum_{i=1}^M \inf_{u\in \sU} \int_{Q_t^{-1}(i)}  \pi_t(dx) c_0(x,u).
\end{equation}
In this context, $\Pi_W$ corresponds to the class of deterministic Markov
control policies. Note that this definition of average cost assumes that the
decoder uses an optimal receiver policy  for encoder
policy $\Pi$ given by \eqref{eq_cdef}; thus our focus is on the encoding operation.

Recall that by Assumption~\ref{AssumptionA} the density $\phi(\,\cdot\,|x)$ of
$f(x,w_t)$ for fixed $x$ is bounded, positive, and Lipschitz, uniformly in
$x$. By \eqref{sourceChannelModel11} and \eqref{filtre}  $\pi_t$ admits a density, which we
also denote by $\pi_t$, given by
\[
\pi_t(z)  = \int_{\R^d}  \phi(z|x_{t-1}) P(dx_{t-1}|
q_{[0,t-1]}), \quad z\in \R^d, \quad t\ge 1.
\]
Thus for any policy $\Pi$, with probability 1 we have $0<\pi_t(z)\le C$ for all
$z$ and $t\ge 1$, where $C$ is an upper bound on $\phi$. Also, if $\phi(z|x)$ is
Lipschitz in $z$  with constant $C_1$ for all $x$, then the bound
\begin{eqnarray*}
\lefteqn{
|\pi_t(z) -\pi_t(z')| } \\
 &\le  & \int_{\R^d}  \bigl|\phi(z|x_{t-1})-\phi(z'|x_{t-1})\bigr|  P(dx_{t-1}|
q_{[0,t-1]}),
\end{eqnarray*}
implies that $\{\pi_t\}_{t\ge 1}$ is uniformly Lipschitz with constant $C_1$.
The collection of all densities with these properties will play an important
part in this paper.

\begin{definition}
\label{defsetS}
Let
$\S$ denote the set of all probability measures on $\R^d$ admitting  densities
that are bounded by $C$ and Lipschitz with constant $C_1$.  
\end{definition}

Note that
viewed as a class of densities, $\S$ is uniformly bounded and equicontinuous.
Lemma~\ref{lemma_sqconv}  in the Appendix shows that $\S$ is closed in
$\P(\R^d)$. Also, the preceding argument implies the following useful lemma.

\begin{lemma}\label{technicalLemma} For any policy $\Pi\in \Pi_W$, we have
$\pi_t \in \S$ for all $t\ge 1$ with probability~1.
\end{lemma}

For technical reasons in most of what follows we restrict the set of
quantizers by only allowing ones that have convex cells. Formally, this
quantizer class $\Q_c$ is defined by
\[
  \Q_c=\{ Q\in \Q: Q^{-1}(i)\subset \R^d \text{ is  convex for $i=1,\ldots,M$}\},
\]
where by convention we declare the empty set convex. Note that each nonempty cell of a
$Q\in \Q_c$ is a convex polytope in $\R^d$. The class of policies $\Pi^C_W$ is obtained by replacing $\Q$ with
$\Q_c$ in Definition~\ref{WVdef}:
\begin{definition} \label{WVCdef}
$\Pi^C_W$  denotes the set of all quantization policies
$\Pi=\{\heta_t\} \in \Pi_W$
such that $\heta_t: \P(\R^d) \to \Q_c$, i.e.,  $Q_t=\heta_t(\pi_t) \in
\mathcal{Q}_c$ for all $t\ge 0$.
\end{definition}

\pagebreak[2]
\begin{remark}  \label{rem_qc} 
\mbox{}
  \begin{itemize}
  \item[(i)] The assumption of convex codecells is adopted for technical
    reasons: the structure of $\Q_c$ detailed below will let us endow it with a
    well-behaved topology.   $\Q_c$ is a fairly powerful class; for
    example, it includes as a proper subset the class of nearest-neighbor
    quantizers considered in \cite{BorkarMitterTatikonda}. Furthermore, it was
    proved in \cite{GyorgyLinder2003} that $\Q_c$ contains all $M$-level optimal
    entropy-constrained quantizers when the source has a density, while the set
    of all $M$-level nearest neighbor quantizers is clearly suboptimal in this
    sense. On the other hand, it is likely that the convex codecell assumption
    results in a loss of system optimality in our case. This can be conjectured
    from the results of \cite{Ant12} where it was shown that in multiresolution
    quantization, requiring that quantizers have convex codecells may preclude
    system optimality even for continuous sources. We note that the
    convex codecell assumption is often made when provably optimal and fast
    algorithms are sought for the design of multiresolution, multiple
    description, and Wyner-Ziv quantizers; see, e.g., \cite{MuEf08}
    and \cite{DuWu09}.

\item[(ii)] As opposed to general quantizers in  $\Q$, any $Q\in \Q_c$ has a parametric
representation. Let such a $Q$ have cells $\{B_1,\ldots,B_M\}$.  As discussed in
\cite{GyorgyLinder2003}, by the separating hyperplane theorem, there exist pairs
of complementary closed half spaces $\{(H_{i,j},H_{j,i}):\, 1\le i,j\le M, i\neq
j\}$ such that  $B_i \subset \bigcap_{j\ne i}
H_{i,j}$ for all $i$. Since $\bar{B}_i\coloneqq \bigcap_{j\ne i} H_{i,j}$ is a closed convex
polytope for each $i$, if  $P\in \P(\R^d)$ admits a density,
then $P(\bar{B}_i\setminus B_i)=0$ for all $i$. We  thus
obtain a $P$-almost sure representation of $Q$ by the $M(M-1)/2$ hyperplanes
$h_{i,j}=H_{i,j}\cap H_{j,i}$. One can represent such a hyperplane $h$ by a
vector $(a_1,\dots,a_d,b)\in \mathbb{R}^{d+1}$ with $\sum_k |a_k|^2=1$ such that
$h = \{x \in \R^d: \sum_i a_i x_i = b\}$, thus obtaining a
parametrization over $\mathbb{R}^{(d+1)M(M-1)/2}$ of all quantizers in
$\mathcal{Q}_c$.
\end{itemize}
\end{remark}

In order to facilitate the stochastic control analysis of the quantization
problem we need an alternative representation of quantizers.  As discussed in,
e.g., \cite{BorkarMitterSahaiTatikonda} and \cite{YukLinSIAM2010}, a quantizer
$Q$ with cells $\{B_1,\ldots,B_M\}$ can also be identified with the stochastic
kernel (regular conditional probability), also denoted by $Q$, from
$\R^d$ to $\sM$ defined by
\[
Q(i|x)= 1_{\{x \in B_i\}}, \quad  i=1,\ldots,M.
\]

We will endow the set of quantizers $\Q_c$ with a topology induced by the
stochastic kernel interpretation. If $P$ is a probability measure on
$\R^d$ and $Q$ is a stochastic kernel from $\R^d$ to
$ \sM$, then $PQ$ denotes the resulting joint probability measure on
$\R^d \times  \sM$ defined through $PQ(dx\,dy)=P(dx)Q(dy|x)$.
For some fixed $P\in \P(\R^d)$ let
\[
\Gamma_P\coloneqq\{PQ\in \P(\R^d\times \sM): Q\in \Q_c\}
\]
It follows from \cite[Thm.~5.8]{YukLinSIAM2010} that $\Gamma_P$ is a compact
subset of $\P(\R^d\times \sM)$ if $P$ admits a density.  If we introduce the
equivalence relation $Q\equiv Q'$ if and only if $PQ=PQ'$, then the resulting
set of equivalence classes, denoted by $(\Q_c)_P$, can be equipped with the
quotient topology inherited from $\Gamma_P$. In this topology $Q_n\to Q$ if and
only if (for representatives of the equivalence classes) $PQ_n\to PQ$
weakly. Also, if $PQ=PQ'$ for $P$ admitting a positive density, then the (convex
polytopal) cells of $Q$ and $Q'$ may differ only in their boundaries, and it
follows that $(\Q_c)_{P}=(\Q_c)_{P'}$ for any $P'$ also admitting a positive
density. From now on we will identify $\Q_c$ with $(\Q_c)_P$ and endow it with
the resulting quotient topology, keeping in mind that this definition does not
depend on $P$ as long as it has a positive density.  Lemma~\ref{lemma_sqconv} in
the Appendix shows that $\Q_c$ is compact. Note that $\pi_t$ for $t\ge 1$ always
has a positive density due to Assumption~\ref{AssumptionA}. However, in some of
the results we will allow $\pi_0$ to violate this assumption (e.g., by letting
$\pi_0$ be a point mass at a given $x_0\in \R^d$). 

For a given policy $\Pi\in \Pi^C_W$, we will consider
$\{(\pi_t,Q_t)\}$ as an $\S\times \Q_c$-valued process.

\section{Existence of optimal policies: Finite horizon setting}\label{sec_finite}

For any quantization policy $\Pi$ in $\Pi_W$ and any  $T\ge 1$ we define
\[
J_{\pi_0}(\Pi,T)\coloneqq \inf_{\gamma} J_{\pi_0}(\Pi,\gamma,T)
= E^{\Pi}_{\pi_0}\biggl[ \frac{1}{T}\sum_{t=0}^{T-1}
  c(\pi_t,Q_t)\biggr],
\]
where $c(\pi_t,Q_t)$ is defined in \eqref{eq_cdef}.

\pagebreak[2]

\begin{assumption}\label{AssumptionACompactness} 
\mbox{}
\begin{itemize}
  \item[(i)] The cost  $c_0: \R^d \times \sU \to \mathbb{R}$ is nonnegative,
    bounded, and continuous.
\item[(ii)] $\sU$ is compact.
\end{itemize}
\end{assumption}

\begin{theorem}\label{ExistenceOfOptimalReceiverPolicy}
Under Assumptions~\ref{AssumptionA} and
  \ref{AssumptionACompactness} an optimal receiver policy always exists, i.e.,
  for any $\Pi\in \Pi_W$ there exist $\gamma=\{\gamma_t\}$ such that
  $J_{\pi_0}(\Pi,\gamma,T) =J_{\pi_0}(\Pi,T)$.

\end{theorem}

\emph{Proof.} At any $t\ge 0$ an optimal receiver has to minimize $\int
P(dx_t|q_{[0,t]}) c_0(x_t,u)$ in $u$. Under Assumption~\ref{AssumptionACompactness},
the existence of a minimizer then follows from a standard argument, see, e.g.,
 \cite[Theorem 3.1]{YukLinSIAM2010}.  \qed

 The following result states the existence of optimal policies in $\Q_c$ for the
 finite horizon setting. The proof is given in
 Section~\ref{ProofOptQuantizerSectionB} of the Appendix.

\begin{theorem}\label{MeasurableSelectionApplies}   Suppose $\pi_0$ admits a
  density or it is a point mass $\pi_0=\delta_{x_0}$ for some $x_0\in \R^d$. For
  any $T\ge 1$, under Assumptions \ref{AssumptionA} and
  \ref{AssumptionACompactness}, there exists a policy $\Pi$ in $\Pi^C_W$ such
  that
\begin{equation}
\label{eq_opt}
J_{\pi_0}(\Pi,T) = \inf_{\Pi' \in \Pi^C_W} J_{\pi_0}(\Pi',T).
\end{equation}
Let $J^T_T(\,\cdot\,)\coloneqq0$ and define $J^T_t(\pi)$ for
$t=T-1,T-2,\ldots,1$, $\pi\in \S$ and    $t=0$, $\pi=\pi_0$, 
by the  dynamic programming recursion
\begin{eqnarray}
\lefteqn{J^T_t(\pi) } \nonumber \\
 &= & \!\! \!\!\!\!\min_{Q \in \mathcal{Q}_c}
\!\!\bigg(  \frac{1}{T}  c(\pi,Q)\! +  \!
E\bigl[J^T_{t+1}(\pi_{t+1})|\pi_t\!=\!\pi,Q_t\!=\!Q\bigr]\!\! \bigg). \label{DPRecursion}
\end{eqnarray}
Then 
\[
J^T_0(\pi_0)=\min_{\Pi\in \Pi^C_W}  J_{\pi_0}(\Pi,T).
\]
\end{theorem}

\section{The finite horizon problem for quadratic cost}\label{LQGCaseSection}

Linear systems driven by Gaussian noise are important in many applications in
control, estimation, and signal processing.  For such linear systems with quadratic
cost (known as LQG optimal control problems), it has been shown that the effect
of the control policies can be decoupled from the estimation error without any
loss (see \cite{YukselAllerton2012}, \cite{TatikondaSahaiMitter} and for a review \cite{YukselBasarBook}). In this section we consider the finite horizon problem under conditions that cover LQG systems. Let $\|x\|$ denote the Euclidean norm of $x\in \R^d$.  We
replace  Assumption~\ref{AssumptionACompactness} of the preceding sections with the
following.

\begin{assumption}\label{AssumptionA'} 
\mbox{}
  \begin{itemize}
  \item[(i)] The function $f$ in the system dynamics
    \eqref{sourceChannelModel11} satisfies $\|f(x,w)\| \le K\bigl(
    \|x\|+\|w\|\bigr) $ for some $K>0$
    and all $x,w\in \R^d$.
\item[(ii)] $\sU=\R^d$ and the cost is given by $c_0(x,u)=\|x-u\|^2$.
  \item[(iii)] The common distribution  $ \nu_w$ of the $w_t$ satisfies $\int \|z\|^2
    \nu_w(dz)<\infty$.
\item[(iv)] $\pi_0$ admits a density such that  $E_{\pi_0}[\|x_0\|^2]<\infty$ or
  it is a point mass $\pi_0=\delta_{x_0}$.
\end{itemize}
\end{assumption}

\begin{remark}  \label{rem3}
\mbox{}
  \begin{itemize}
\item[(i)] The above conditions cover the case of a linear-Gaussian system
\[
x_{t+1} = Ax_t + w_t, \quad t =0, 1, 2,\ldots,
\]
where $\{w_t\}$ is an i.i.d.\ Gaussian noise sequence with zero mean, $A$ is a
square matrix, and $\pi_0$ admits a Gaussian density having zero mean.

\item[(ii)] Assumption~\ref{AssumptionA'}(i)
    implies
\[
\|x_{t}\|^2 \le  \hat{K}  \biggl( \|x_0\|^2 + \sum_{i=0}^{t-1}
  \|w_{i}\|^2\biggl)
\]
for some $\hat{K}$ that depends on $t$ (see \eqref{eq_xtbound}). Together with
Assumptions~\ref{AssumptionA'}(iii) and (iv), this implies
$E_{\pi_0}\bigl[\|x_t\|^2\bigr]<\infty $ for all $t\ge 0$ under any quantization
policy . Therefore
\[
 \int_{\R^d} \|x_t\|^2P(dx_t|q_{[0,t]})\, dx<\infty
\]
and  an
optimal receiver  policy exists and is given by
\begin{equation}
\label{eq_gammaopt}
\gamma_t(q_{[0,t]})=\int_{\R^d} x_t  P(dx_t|q_{[0,t]}).
\end{equation}

\end{itemize}
\end{remark}

The following is a restatement of Theorem~\ref{MeasurableSelectionApplies} under
conditions that  allow unbounded cost.  The proof is relegated to
Section~\ref{MeasurableApplies2Proof} of the Appendix.

\begin{theorem}\label{MeasurableSelectionApplies2}
Under Assumptions  \ref{AssumptionA} and  \ref{AssumptionA'}, for any $T\ge 1$
there exists an optimal policy in $\Pi^C_W$ in the sense of \eqref{eq_opt} and
the dynamic programming recursion  \eqref{DPRecursion} for $J^T_t(\pi_t)$ also
holds.
\end{theorem}

\section{Infinite horizon setting} \label{sec_infinite}

For the infinite horizon setting, one may consider the discounted cost problem
where the goal is to find policies that achieve
\begin{eqnarray}\label{LQGopt22}
V^{\beta}(\pi_0)=\inf_{\Pi\in \Pi_W^C} J_{\pi_0}^{\beta}(\Pi)
\end{eqnarray}
for some $\beta \in (0,1)$, where
\[J^{\beta}_{\pi_0}(\Pi)  = \inf_{\gamma} \lim_{T \to \infty} E^{\Pi,\gamma}_{\pi_0}\biggl[\, \sum_{t=0}^{T-1} \beta^t c_0(x_t,u_t)\biggr].
\]

The existence of optimal policies for this problem follows from the results
in the previous section. In particular, it is well known that the value
iteration  algorithm (see, e.g., \cite{CTCN}) will converge to an optimal
solution, since the cost function is bounded and the measurable selection
hypothesis is applicable in view of Theorem
\ref{MeasurableSelectionApplies}. This leads to the fixed point
equation
\begin{eqnarray*}
\lefteqn{V^{\beta}(\pi)} \\
& = & \!\!\!\!\min_{Q \in {\cal Q}_c}\!\! \bigg(\!c(\pi,Q) + \beta\!\! \int_{\R^d} \!\!\!\!
P(d\pi_{t+1}|\pi_t\!=\!\pi,Q_t\!=\!Q) V^{\beta}(\pi_{t+1})  \bigg). \nonumber
\end{eqnarray*}

The more challenging case is the average cost problem where one considers
\begin{equation}
\label{infiniteCost2}
J_{\pi_0}(\Pi)  = \inf_{\gamma} \limsup_{T \to \infty}
E^{\Pi,\gamma}_{\pi_0}\biggl[\frac{1}{T}   \sum_{t=0}^{T-1} c_0(x_t,u_t)\biggr]
\end{equation}
and the goal is to find an optimal policy attaining
\begin{eqnarray}\label{LQGoptAVG}
J_{\pi_0} \coloneqq   \inf_{\Pi\in \Pi_A}  J_{\pi_0}(\Pi) .
\end{eqnarray}

For the infinite horizon setting the structural results in
Theorems~\ref{witsenhausenTheorem} and \ref{WalrandVaraiyaTheorem} are not
available in the literature, due to the fact that the proofs are based on
dynamic programming which starts at a finite terminal time stage and optimal
policies are computed backwards. Recall that $\pi^*\in \P(\R^d)$ is called an
invariant measure for $\{x_t\}$ if setting $\pi_0=\pi^*$ results in $P(x_t\in B)
= \pi^*(B)$ for every  $t$ and Borel set $B$ (in this case $\{x_t\}$ is a 
strictly stationary process). The next result proves an infinite-horizon analog of
Theorem~\ref{WalrandVaraiyaTheorem} under the assumption that an invariant
measure $\pi^*$ for $\{x_t\}$ exists and $\pi_0=\pi^*$.

\begin{theorem}
\label{thm_QWopt}
Assume the cost $c_0$ is bounded and an invariant measure $\pi^*$ exists. If
$\{x_t\}$ starts from $\pi^*$, then there
exists an optimal policy in $\Pi_W$ solving the minimization problem
(\ref{LQGoptAVG}), i.e., there exists $\Pi\in \Pi_W$ such that
\[
\limsup_{T \to \infty}  E^{\Pi}_{\pi^*}\biggl[ \frac{1}{T}  \sum_{t=0}^{T}
 c(\pi_t,Q_t)\biggr] = J_{\pi^*}.
\]
\end{theorem}

The proof of the theorem relies on a construction that pieces together
policies from $\Pi_W$ that on time segments of appropriately large lengths
increasingly well approximate the minimum infinite-horizon cost achievable by policies in
$\Pi_A$. Since the details are somewhat tedious, the proof is relegated to
Section~\ref{Opt_PW_policy} of the Appendix. We note that the condition that
$c_0$ is bounded is not essential and, for example, the theorem holds for the
quadratic cost if the invariant measure has a finite second moment.

\begin{remark} \label{rem_approxopt} 
\mbox{}
\begin{itemize}
\item[(i)] If the source is a positive Harris recurrent Markov chain
  \cite{MeynBook}, then the policy constructed in the proof of
  Theorem~\ref{thm_QWopt} achieves the optimal average cost corresponding to the
  stationary source even when the chain is not started from the invariant
  distribution $\pi^*$. This can be shown by inspecting the details of the proof
  and using the continuity of the \emph{value function} as stated by
  Theorem~\ref{thmexistencehypothesis} in the Appendix, combined with the fact
  that  $P(x_t\in \,\cdot \,)\to \pi^* $ in
  total variation   as $t\to \infty$ for any initial distribution $\pi_0$. 

\item[(ii)] The proof of Theorem~\ref{thm_QWopt} demonstrates that for every
  $\epsilon>0$  there exists a finite memory encoding policy whose performance is
  within $\epsilon$ of the optimum value $J_{\pi^*}$.  This scheme can be
  computed using finite horizon dynamic programming, giving the result
   practical relevance. 
\end{itemize}
\end{remark}

The optimal policy constructed in the proof of Theorem~\ref{thm_QWopt} may not
be stationary. In general, a stationary policy in a given class of policies is
called optimal if it performs as well as any other policy in that class. In the
next section we establish the existence of an optimal stationary policy in
$\Pi^C_W$ if randomization is allowed.

\subsection{Classes of randomized quantization policies}

We will consider two classes of randomized policies.

\emph{Randomized Walrand-Varaiya-type (Markov) policies:} These policies,
denoted by $\bar{\Pi}_W^C$, are randomized over $\Pi^C_W$, the
Walrand-Varaiya-type Markov policies with quantizers having convex cells
(Definition~\ref{WVCdef}).  Each $\Pi\in \bar{\Pi}_W^C$ consists of a sequence
of stochastic kernels $\{\bar{\eta}_t\}$ from $ \P(\R^d)$ to $\Q_c$. Thus, under
$\Pi$, for any $t\ge0$,
\begin{eqnarray*}
\lefteqn{P^{\Pi}\bigl(Q_t(x_t)=q_t | q_{[0,t-1]}, Q_{[0,t-1]}, \pi_{[0,t]}\bigr)} \nonumber
\qquad\qquad \\
 &=&\!\!\!\int_{\mathcal{Q}_c} \biggl( \int_{\R^d}
 1_{\{Q(x)=q_t\}} \pi_t(dx)\biggr) \bar{\eta}(dQ|\pi_t). 
\end{eqnarray*}

It follows from, e.g., \cite{Gikhman} or \cite{SLY}  that an equivalent model
for randomization 
can be obtained by considering an i.i.d.\ randomization sequence $\{r_t\}$,
independent of $\{x_t\}$ and uniformly distributed on $[0,1]$, and a sequence of
(measurable) randomized encoders $\{\heta_t\}$ of the form $\heta_t:
\P(\R^d)\times[0,1] \to \Q_c$ and $Q_t$ such that  $Q_t=\heta_t(\pi_t,r_t)$. In this case the
induced stochastic kernel encoder  $\bar{\eta}_t$ is determined by
\[
\bar{\eta}_t(D|\pi_t)= u\bigl\{ r: \heta(\pi_t,r)\in D \bigr\}
\]
for any Borel subset $D$ of $\Q_c$, where $u$ denotes the uniform distribution
on $[0,1]$. For randomized policies we assume that all the randomization
information is shared between the encoder and the decoder, that is
\[
I^r_t\coloneqq (q_{[0,t-1]}, r_{[0,t-1]})
\]
is known at the decoder which can  therefore track  $\pi_t$ given by
\[
\pi_t(A) \coloneqq P(x_t\in A | q_{[0,t-1]}, r_{[0,t-1]})
\]
for any Borel set $A\subset \R^d$.

We note that the cost $c(\pi_t,Q_t)$ is still defined by \eqref{eq_cdef} since
the decoder, having access to $I^r_t$ can also track $Q_t$. Also, in computing the
cost $ E^{\Pi}_{\pi_0}\Bigl[\frac{1}{T} \sum_{t=0}^{T} c(\pi_t,Q_t)\Bigr]$
of policy $\Pi\in \bar{\Pi}_W^C$  after $T$ time stages,  the
expectation is also taken with respect to the randomization sequence $\{r_t\}$.

\emph{Randomized stationary Walrand-Varaiya-type (Markov) policies:} Denoted by
$\bar{\Pi}_{W,S}^C$, this class consists of all policies in $\bar{\Pi}_W^C$ that
are stationary, i.e., the stochastic kernels $\bar{\eta}_t$ or the randomized encoders
$\heta_t$ do not depend on the time index $t$.

\subsection{Existence of optimal  stationary policies}

\subsubsection{The bounded cost case}
\label{sec_infhor_bounded}

In the infinite horizon setting, we add the following assumption, in
addition to Assumptions \ref{AssumptionA} and \ref{AssumptionACompactness}.
\begin{assumption}\label{AssumptionB} 
The chain $\{x_t\}$ is positive Harris recurrent (see \cite{MeynBook}) with unique invariant measure $\pi^*$ such that for all $x_0\in \R^d$,
\[
\lim_{t \to \infty} E_{\delta_{x_0}}\bigl[\|x_t\|^2\bigr] = \int_{\R^d}
\|x\|^{2} \pi^*(dx)< \infty.
\]
\end{assumption}

\begin{remark} A sufficient condition for Assumption~\ref{AssumptionB} to hold
  is that $f(x,w)$ in (\ref{sourceChannelModel11}) is continuous in $x$ and
  satisfies $\|f(x,w)\|\le K\bigl(\|x\|+\|w\|\bigr)$ for some $K<1$, and $w_t$
  has zero mean and second moment $E\bigl[\|w_t\|^2\bigr] <\infty$. This follows
  since the upper bound on $f$ and a  straightforward calculation imply that
  the drift condition \cite{MeynBook}
\[
E[V(x_{t+1})|x_t=x] \leq V(x) - g(x) + b1_{\{x \in C\}}
\]
holds with  $V(x)=\|x\|^2$, $g(x)=(1-(K^2+\epsilon))\|x\|^2$,  $0 <
  \epsilon < 1 - K^2$, and  $C=\{\|x\| \leq M\}$ (a compact set), where
\[
M=\frac{K^2E[\|w\|]+\sqrt{(K^2E[\|w\|])^2 + K^2\epsilon E[\|w\|^2] }}{\epsilon}
\] 
and $b= K^2(E[\|w\|^2]+ 2 E[\|w\|]M)$.  The continuity of $f(x,w)$ in $x$
implies that the chain is weak Feller \cite{MeynBook}. This and the drift
condition imply through \cite[Theorem 2.2]{YukMeynTAC2013} that there exists an
invariant probability measure with a finite second moment. The irreducibility
and aperiodicity \cite{MeynBook} of the chain under Assumption~\ref{AssumptionA}
implies the uniqueness of the invariant probability measure and positive Harris
recurrence, leading to Assumption~\ref{AssumptionB}.
  \end{remark}

To show the existence of an optimal stationary policy, we adopt the convex
analytic approach of \cite{Borkar} (see \cite{survey} for a detailed
discussion). Here we only present the essential steps.

Fix a policy $\Pi\in \bar{\Pi}_W^C$ and an initial distribution $\pi_0$ .  Let
$v_t\in \P(\P(\R^d)\times \Q_c) $ be the sequence of  expected  occupation measures
determined by
\[
v_t(D) = E_{\pi_0}^{\Pi} \biggl(  \frac{1}{t}  \sum_{i=0}^{t-1} 1_{\{(\pi_i,Q_i)
 \in D\}}  \biggr)
\]
for any Borel subset $D$ of $\P(\R^d) \times \Q_c$.

Let $P(d\pi_{t+1} | \pi_t,Q_t)=P^{\Pi}(d\pi_{t+1} | \pi_t,Q_t) $ be the
transition kernel determined by the filtering equation $\eqref{filtre}$ and note
that it does not depend on $\Pi$ and~$t$. Also note that $P(\S | \pi,Q)=1$ for
any $\pi$ and $Q$, where $\S \subset \P(\R^d)$ is the set of probability
measures, defined in Definition~\ref{defsetS}, which admit densities that
satisfy the same upper bound and Lipschitz condition as the density of the
additive noise $w_t$ ($\S$ contains the set of reachable states for $\{\pi_t\}$
under any quantization policy).

If $\sX$ is a topological space, let $\C_b(\sX)$ denote the set of all bounded
and continuous real-valued functions on $\sX$.
Let   $\G$  be the set of so-called ergodic occupation measures on
$\P(\R^d)  \times \Q_c$, defined by
\begin{eqnarray*}
\G &= & \biggl\{v \in \P(\P(\R^d)  \times \Q_c)  :
 \int f(\pi) v(d\pi\, dQ) \\
& &\!\!\!\!\!\!\!\! \!\!\!\!  =\iint f(\pi') P( d\pi'|
\pi, Q) v(d\pi\, dQ)    \text{ for all $f\in \C_b(\P(\R^d))$}   \biggr\} .
\end{eqnarray*}
Note that any $v\in \G$ is supported on $\S\times \Q_c$.

Any $v\in \G$ can be disintegrated as $v(d\pi\, dQ) =
\hat{v}(d\pi)\bar{\eta}(dQ|\pi)$, where $\bar{\eta}$ is a stochastic kernel from
$\P(\R^d)$  to $\Q_c$ which  corresponds to the randomized stationary policy
$\Pi=\{\bar{\eta}_t\}$ in $\bar{\Pi}^C_{W,S}$ such that $\bar{\eta}_t=\bar{\eta}$
for all $t$. Then the transition kernel of the process $\{(\pi_t,Q_t)\}$  induced by  $\Pi$ does not depend on $t$ and is
given by
\[
P^{\Pi}( d\pi_{t+1} \, dQ_{t+1}| \pi_t, Q_t)  = P( d\pi_{t+1}| \pi_t, Q_t)
\bar{\eta}(dQ_{t+1}|\pi_t).
\]
In fact, it directly follows from the definition of $\G$ that
\begin{eqnarray}
\lefteqn{\int g(\pi,Q) v(d\pi\, dQ) } \nonumber \\
&=& \int \int g(\pi',Q') P^{\Pi}( d\pi' \, dQ'|
\pi, Q) v(d\pi\, dQ) \label{eq_vinv}
\end{eqnarray}
for all $g\in  \C_b(\P(\R^d)\times \Q_c )$, i.e., $v$ is
an invariant measure for the transition kernel $P^{\Pi}$.

The following proposition, proved in Section~\ref{StationaryOptimalProof}, will
imply the existence of optimal stationary policies.

\begin{proposition}\label{weakLimitTheorem}
\begin{itemize}
\item[\rm (a)]  For any initial distribution $\pi_0$ and policy  $\Pi\in \bar{\Pi}_W^C$,
if $\{v_{t_n}\}$ is a subsequence of the expected occupation measures  $\{v_t\}$ such that $v_{t_n}\to \bar{v}$ weakly, then
$\bar{v}\in \G$. Furthermore
\begin{eqnarray}
\lefteqn{ \lim_{n\to  \infty} \int_{\P(\R^d) \times \Q_c }c(\pi,Q)
  v_{t_n}(d\pi\, dQ)} \qquad \quad  \nonumber \\
&  = &
\int_{\P(\R^d) \times \Q_c }c(\pi,Q) \bar{v}(d \pi\,  dQ).   \label{eq_cvtconv}
\end{eqnarray}
\item[(b)]  For any  $x_0\in \R^d$, initial distribution $\pi_0=\delta_{x_0}$,
  and policy $\Pi\in \bar{\Pi}_W^C$, $\{v_t\}$ is relatively compact.
\item[\rm (c)] $\G$ is compact.
\end{itemize}
\end{proposition}

For any initial distribution $\delta_{x_0}$ and policy $\Pi\in
\bar{\Pi}_{W}$, we have
\begin{eqnarray*}
\lefteqn{ \liminf_{T \to \infty} E_{\delta_{x_0}}^{\Pi}  \biggl[\frac{1}{T}
\sum_{t=0}^{T-1} c(\pi_t,Q_t) \biggr] } \qquad \\
& = & \liminf_{T\to  \infty} \int_{\P(\R^d) \times \Q_c }c(\pi,Q)
v_T(d \pi\,  dQ).
\end{eqnarray*}
Let
$\{v_{T_n}\}$  be a subsequence of $\{v_T\}$ such that
\begin{eqnarray*}
\lefteqn{ \liminf_{T\to  \infty} \int_{\P(\R^d) \times \Q_c }c(\pi,Q)
v_T(d \pi\,  dQ) } \qquad \\
&= &  \lim_{n\to  \infty} \int_{\P(\R^d) \times \Q_c }c(\pi,Q)
v_{T_n}(d \pi\,  dQ).
\end{eqnarray*}
By Proposition~\ref{weakLimitTheorem}(b) there exists a subsequence
of $\{v_{T_n}\}$, which we  also denote by $\{v_{T_n}\}$,  weakly converging to
some $\bar{v}$. By Proposition~\ref{weakLimitTheorem}(a) we have $\bar{v}\in \G$ and  $\int c\, dv_{T_n} \to \int c\, d\bar{v}$.  Therefore
\begin{eqnarray*}
  \liminf_{T \to \infty} E_{\delta_{x_0}}^{\Pi}  \biggl[\frac{1}{T}
\sum_{t=0}^{T-1} c(\pi_t,Q_t) \biggr]  \!\!\!
&=  &  \!\!\!
\int_{\P(\R^d) \times \Q_c }  \!\!\! c(\pi,Q) \bar{v}(d \pi\,  dQ)\\
&\ge &  \!\!\!  \inf_{v\in \G} \int_{\P(\R^d) \times \Q_c } \!\!\!  \!\!\!
\!\!\! c(\pi,Q)v(d \pi\,  dQ). 
\end{eqnarray*}
In addition, since $c$ is continuous on $\S\times \Q_c$ (by
Lemma~\ref{continuityofc}) and each  $v\in \G$ is supported on $\S\times \Q_c$,
the mapping
$v\mapsto \int c\, dv$ is continuous on $\G$. Since  $\G$ is compact
by Proposition~\ref{weakLimitTheorem}(c), there exists $v^*\in \G$
achieving the above infimum.  Hence
\begin{eqnarray}
c^* &\coloneqq  & \int_{\P(\R^d) \times \Q_c }c(\pi,Q) v^*(d \pi\,  dQ)
\nonumber\\ &= &  \min_{v\in \G} \int_{\P(\R^d) \times \Q_c }c(\pi,Q)
v(d\pi,dQ) \label{eq_optvstar} 
 \end{eqnarray}
provides an ultimate lower bound on the infinite-horizon cost of any policy.

The following theorem shows the existence of a stationary policy achieving this
lower bound if we consider the initial distribution $\pi_0$ as a ``design
parameter'' we can freely choose.

\begin{theorem}\label{StationaryOptimal}
  Under Assumptions \ref{AssumptionA}, \ref{AssumptionACompactness} and
  \ref{AssumptionB}, there exists a stationary policy $\Pi^*$  in
  $\bar{\Pi}^C_{W,S}$ that is optimal in the sense that with an appropriately
  chosen initial distribution $\pi_0^*$,
\[
\lim_{T \to \infty} E_{\pi_0^*}^{\Pi^*}  \biggl[\frac{1}{T}
\sum_{t=0}^{T-1} c(\pi_t,Q_t) \biggr] \le  \liminf_{T \to \infty} E_{\delta_{x_0}}^{\Pi}  \biggl[\frac{1}{T}
\sum_{t=0}^{T-1} c(\pi_t,Q_t) \biggr]
\]
for any $x_0\in \R^d$ and  $\Pi\in \bar{\Pi}^C_{W}$.
\end{theorem}

\begin{proof} We must prove the existence of $\Pi^* \in \bar{\Pi}^C_{W,S}$ which
  achieves infinite horizon cost $c^*$ for some initial distribution $\pi_0^*$.
  Consider $v^*$ achieving the minimum in \eqref{eq_optvstar}, disintegrate it
  as $v^*(d\pi\, dQ) = \hat{v}^*(d\pi)\bar{\eta}^*(dQ|\pi)$, and let $\Pi^* \in
  \bar{\Pi}^C_{W,S}$ be the policy corresponding to $\bar{\eta}^*$. Since $v^*$
  is an invariant measure for the transition kernel $P^{\Pi^*}$ (see
  \eqref{eq_vinv}), for any $T\ge 1$,
\[
c^* =
E_{\hat{v}^*}^{\Pi^*} \biggl[ \frac{1}{T} \sum_{t=0}^{T-1}c(\pi_t,Q_t)\biggr],
\]
where the notation $E_{\hat{v}^*}^{\Pi^*}$ signifies that the initial
distribution $\pi_0$ is picked randomly with distribution $\hat{v}^*$.
Thus
\begin{equation}
  \label{eq_vstar}
c^* = \lim_{T \to \infty}E_{\hat{v}^*}^{\Pi^*} \biggl[ \frac{1}{T}
\sum_{t=0}^{T-1}c(\pi_t,Q_t)\biggr].
\end{equation}
From  the individual ergodic theorem (see
\cite{HernandezLermaLasserre}) the limit
\[
f(\pi_0)\coloneqq \lim_{T \to \infty}E_{\pi_0}^{\Pi^*} \biggl[ \frac{1}{T}
\sum_{t=0}^{T-1}c(\pi_t,Q_t)\biggr]
\]
exists for $\hat{v}^*$-a.e.\ $\pi_0$ and
\[
 \int_{\P(\R^d)} f(\pi_0)\hat{v}^*(d\pi_0) =c^*.
\]
Hence for some $\pi_0$ in the support of $\hat{v}^*$ we must have
\[
\lim_{T \to \infty}E_{\pi_0}^{\Pi^*} \biggl[ \frac{1}{T}
\sum_{t=0}^{T-1}c(\pi_t,Q_t)\biggr]\le c^*
\]
Any such $\pi_0$ can be picked as $\pi_0^*$
so that the claim of the theorem holds.
\end{proof}

In the preceding theorem the initial state distribution $\pi_0$ is a design
parameter which is chosen along with the quantization policy to
optimize the cost. This assumption may be unrealistic. However, consider the
fictitious optimal stationary policy in \eqref{eq_vstar} which is allowed to
pick the initial distribution $\pi_0$ according to $\hat{v}^*$. It follows from
the analysis in the proof of Proposition~\ref{weakLimitTheorem} (see
\eqref{eq_vinariant}) that the expectation of $\pi_0$ according to $\hat{v}^*$
is precisely the invariant distribution $\pi^*$ for $\{x_t\}$.  Based on this,
one can prove the following, more realistic version of the optimality
result. The proof, which is not given here, is an expanded  and more
refined version of the proof of
Theorem~\ref{StationaryOptimal}.

\begin{theorem}\label{deterministicOptimal2} Under the setup of
  Theorem~\ref{StationaryOptimal}, assume that $\{x_t\}$ is started from the
  invariant distribution $\pi^*$. If the optimal stationary policy $\Pi^* \in
  \bar{\Pi}^C_{W,S}$ is used in such a way that the encoder and decoder's
  initial belief $\pi_0$ is picked randomly according to $\hat{v}^*$ \emph{(}but
  independently of $\{x_t\}\, $\emph{)}, then $\Pi^*$ is still optimal in the
  sense of Theorem~\ref{StationaryOptimal}.
\end{theorem}

\begin{remark}  We have not shown that an optimal stationary policy is
  deterministic. In the convex analytic approach, the existence of an optimal
  deterministic stationary policy directly follows if one can show that the
  extreme points of ergodic occupation measures satisfy the following: (i) They
  are induced by deterministic policies; and (ii) under these policies the state
  invariant measures are ergodic. This property of the extreme points of the set
  of ergodic occupation measures has been proved  by Meyn in
  \cite[Proposition~9.2.5]{CTCN} for countable state spaces and by Borkar in
  \cite{Borkar} and \cite{ABG} for a specific case involving $\R^d$ as
  the state space and a non-degeneracy condition which amounts to having a
  density assumption on the one-stage transition kernels. Unfortunately, these
  approaches do not seem to apply in our setting.
\end{remark}

\subsubsection{The quadratic cost case}
In the infinite horizon setting for the important case of the (unbounded)
quadratic cost function, we add the following assumption, in addition to
Assumption~\ref{AssumptionA'}.
\begin{assumption}\label{AssumptionB'}
The chain $\{x_t\}$ is positive Harris recurrent with unique invariant measure
$\pi^*$ such that  for some $\epsilon > 0$ and  all $x_0\in \R^d$,
\[
\lim_{t \to \infty} E_{\delta_{x_0}}\bigl[\|x_t\|^{2+\epsilon}\bigr] =
\int_{\R^d} \|x\|^{2+\epsilon} \pi^*(dx)< \infty.
\]
\end{assumption}

\begin{remark}  A sufficient condition for Assumption~\ref{AssumptionB'} to
  hold is that $f$ in (\ref{sourceChannelModel11}) satisfies $\|f(x,w)\|\le
  K\bigl(\|x\|+\|w\|\bigr)$ for some $K<1$ and $w_t$ has zero mean and finite
  $(2+\epsilon)$th
  moment $E\bigl[\|w_t\|^{(2+\epsilon)}\bigr] <\infty$. In particular, the
  assumption holds  for the
  LQG case $x_{t+1}=Ax_t + w_t$, with $A$ being a $d\times d$ matrix having
  eigenvalues of absolute value less than $1$ and $w_t$ having a nondegenerate
  Gaussian distribution with zero mean.
\end{remark}

\begin{theorem}\label{StationaryOptimalQuadratic}
  Under Assumptions \ref{AssumptionA'} and \ref{AssumptionB'}, there exists a stationary policy $\Pi^*$  in
  $\bar{\Pi}^C_{W,S}$ that is optimal in the sense that with an appropriately
  chosen initial distribution $\pi_0^*$,
\[
\lim_{T \to \infty} E_{\pi_0^*}^{\Pi^*}  \biggl[\frac{1}{T}
\sum_{t=0}^{T-1} c(\pi_t,Q_t) \biggr] \le  \liminf_{T \to \infty} E_{\delta_{x_0}}^{\Pi}  \biggl[\frac{1}{T}
\sum_{t=0}^{T-1} c(\pi_t,Q_t) \biggr]
\]
for any $x_0\in \R^d$ and $\Pi\in \bar{\Pi}^C_{W}$. Furthermore, if $\{x_t\}$ is
started from the invariant distribution $\pi^*$ and  the optimal stationary
policy $\Pi^* \in \bar{\Pi}^C_{W,S}$ is used in such a way that the encoder and
decoder's initial belief $\pi_0$ is picked randomly according to $\hat{v}^*$
\emph{(}but independently of $\{x_t\}\, $\emph{)}, then $\Pi^*$ is still optimal
in the above sense (with $\pi^*$ replacing $\pi_0^*$).
\end{theorem}

\begin{proof} The proof is almost identical to that of
  Theorems~\ref{StationaryOptimal} and \ref{deterministicOptimal2}, with the
  following minor adjustments, which are needed to accommodate the unboundedness
  of the quadratic cost function. This modification is facilitated by
  Assumption~\ref{AssumptionB'} which implies that, similar to
  \eqref{eq_vttight} and \eqref{eq_moments} in the proof of
  Proposition~\ref{weakLimitTheorem}, for the sequence  of
  expected occupation measures $\{v_t\}$ corresponding to any initial distribution
  $\delta_{x_0}$, we have
\[
\sup_{t\ge 0}  \int_{\P(\R^d)\times \Q_c}  \biggl( \int_{\R^d}  \|x\|^{2+\epsilon}\pi(dx)
\biggr) v_t(d\pi\, dQ) <\infty,
\]
as well as for all $v\in \G$,
\begin{eqnarray*}
  \lefteqn{ \int_{P(\R^d)\times \Q_c} \biggl(\int_{\R^d}
\|x\|^{2+\epsilon}\pi(dx) \biggl) v(d\pi\, dQ)  }\qquad \qquad  \\
&= & \int_{\R^d} \|x\|^{2+\epsilon}\pi^*(dx) <\infty.
\end{eqnarray*}
These uniform integrability properties of $\{v_t\}$ and $\G$ allow us to  use the
continuity result Lemma~\ref{ContinuityofOptimalValue1} for $c(\pi,Q)$. All
other parts of the proof remain unchanged. \end{proof}

\section{Concluding remarks}
\label{sec_conclusions}

In this paper we established structural and existence results concerning optimal
quantization policies for Markov sources. The key ingredient of our analysis was
the characterization of quantizers as a subset of the space of stochastic
kernels.  This approach allows one to introduce a useful topology with respect
to which the set of quantizers with a given number of convex codecells is
compact, facilitating the proof of existence results. We note that both our
assumption of convex-codecell quantizers and the more restrictive assumption of
nearest neighbor-type quantizers in Borkar \emph{et al$.$}
\cite{BorkarMitterTatikonda} may preclude global optimality over all zero-delay 
quantization policies. The existence and finer structural characterization of
such globally optimal policies are still open problems.

The existence and the structural results can be useful for the design of
networked control systems where decision makers have imperfect observation of a
plant to be controlled. The machinery presented here is particularly useful in
the context of optimal quantized control of a linear system driven by unbounded noise: For
LQG optimal control
problems  it has been shown that the effect of the control policies can be
decoupled from the estimation error and the design results here can be used to
establish existence of optimal quantization and control policies for LQG
systems.

The approach developed in this paper can also be applied to the case where
$\{x_t\}$ is a Markov chain with finite state space $\sX$. In this case stronger
results can be obtained with significantly less technical complications. In
particular, when the state space is finite one does not need the convex codecell
assumption since there are only a finite number of $M$-level quantizers on
$\sX$. Also, the global optimality of Walrand-Varaiya type policies for the
infinite horizon discounted cost problem can be easily proved if $\sX$ is
finite.  In addition, one can prove the optimality of \emph{deterministic}
stationary policies for the average cost problem 
under the irreducibility condition  $P(x_{t+1}=b|x_t=a) > 0$ for all
$a, b \in \sX$. Similar to \cite{BorkarMitterTatikonda}, such an optimality
result follows from a \emph{vanishing discount} argument (see, e.g.,
\cite[Theorem 5.2.4]{HernandezLermaMCP}) using arguments similar to Lemmas~4.1
and 4.2 in \cite{BorkarMitterTatikonda} and the fact that the filtering process
forgets its initial state exponentially fast under the irreducibility condition.

A further research direction is the formulation of the communication problem
over a channel with feedback. The tools and the topological
analysis developed in this paper could be useful in establishing optimal coding
and decoding policies  and the derivation of error-exponents with
feedback. Relevant efforts in the literature on this topic include
\cite{TatikondaMitter}.

\section{Appendix}

\subsection{Auxiliary results}\label{ProofOptQuantizerA}

Recall that a sequence of probability measures $\{\mu_n\}$ in $\P(\sX)$ converges to $\mu\in \mathcal{P}(\sX)$ \emph{weakly} if $
\int_{\sX} c(x) \mu_n(dx) \to \int_{\sX}c(x) \mu(dx)$ for
every continuous and bounded $c: \sX \to \R$.
For  $\mu, \nu \in \P(\sX$)
the \emph{total variation} metric is defined by
\begin{eqnarray}
\!\!\!\!\! d_{TV}(\mu,\nu)\!\!\!\! &\coloneqq & \!\!\!\!2 \sup_{B \in \B(\sX))}
|\mu(B)-\nu(B)|\nonumber  \\* 
 &=&  \!\!\!\! \!\!\! \!\sup_{g: \, \|g\|_{\infty} \leq 1} \bigg| \int
g(x)\mu(dx) -\!\! \int g(x)\nu(dx) \bigg|,
\label{TValternative}
\end{eqnarray}
where the second supremum is over all measurable real functions $g$ such that
$\|g\|_{\infty} \coloneqq \sup_{x \in \sX} |g(x)|\le 1$.

\begin{definition}[\cite{YukLinSIAM2010}] \label{defnConvPaper} Let $P\in
  \P(\R^d)$. A quantizer sequence $\{Q_n\}$ converges to $Q$ weakly at $P$
  \emph{(}$Q_n \to Q$ weakly at $P$\emph{)} if $PQ_n \to PQ$ weakly. Similarly,
  $\{Q_n\}$ converges to $Q$ in total variation at $P$ \emph{(}$Q_n \to Q$ in total
  variation at $P$\emph{)} if $PQ_n \to PQ$ in total variation.
\end{definition}

The following lemma will be very useful in the upcoming optimality proofs.

\pagebreak[2]
\begin{lemma}\label{keyTechnicalLemma}
\begin{itemize}
\item[\em (a)] Let $\{\mu_n\}$ be a sequence of probability density functions on
  $\R^d$  which are  uniformly equicontinuous and uniformly
  bounded  and assume $\mu_n\to \mu$ weakly. Then $\mu_n\to \mu$  in
  total variation.

\item[\em (b)] Let $\{Q_n\}$ be a sequence in $\mathcal{Q}_c$ such that $Q_n \to
  Q$ weakly at $P$ for some $Q\in \mathcal{Q}_c$. If $P$ admits a density, then
  $Q_n \to Q$ in total variation at $P$. If the density of  $P$ is positive, then
  $Q_n \to Q$ in total variation at \emph{any} $P'$ admitting a density.

\item[\em (c)]  Let $\{Q_n\}$ be a sequence in $\mathcal{Q}_c$ such that $Q_n \to
  Q$ weakly at $P$ for some $Q\in \mathcal{Q}_c$ where  $P$ admits a positive
  density.  Suppose
  further that $P'_n \to P'$ in total variation where $P'$ admits a
  density. Then $P'_nQ_n \to P'Q$
  in total variation.

\end{itemize}
\end{lemma}

\smallskip \emph{Proof.} (a) We will denote a density and its induced
probability measure by the same symbol. By the Arzel\`a-Ascoli theorem
 the sequence of densities $\{\mu_n\}$, when restricted to a given
compact subset of $\R^d$, is relatively compact with respect to the supremum
norm. Considering the sequence of increasing closed balls $K_i=\{x:\|x\|\le i\}$
of radius $i=1,2,\ldots$, one can use Cantor's diagonal argument as in
\cite[Lemma 4.3]{YukLinSIAM2010} to obtain a subsequence $\{\mu_{n_k}\}$ and a
nonnegative continuous function $\hat{\mu}$ such that $\mu_{n_k}(x)\to
\hat{\mu}(x)$ for all $x$, where the convergence is uniform over compact
sets. Since $\int_B | \mu_{n_k}(x) -\hat{\mu}(x)|\, dx \to 0 $ for any bounded
Borel set $B$, and since $\{\mu_n\}$ is tight by weak convergence, it follows
that $\hat{\mu}$ is a probability density. Since $\mu_{n_k}$ converges to
$\hat{\mu}$ pointwise, by Scheffe's theorem \cite{Bil86} $\mu_{n_k}$ converges
to $\hat{\mu}$ in the $L_1$ norm, which is equivalent to convergence in total
variation.  Since $\mu_n\to \mu$ weakly, we must have $\mu=\hat{\mu}$.

The preceding argument implies that any subsequence of $\{\mu_n\}$ has a further
subsequence that converges to $\mu$ in (the metric of) total variation. This
implies that $\mu_n\to \mu$ in total variation.

(b)
It was shown in the proof of Theorem~5.7 of  \cite{YukLinSIAM2010} that
\begin{eqnarray}
d_{TV}(PQ_n,PQ)  &\le &  \sum_{i=1}^M P(B^n_i \bigtriangleup B_i), \label{diff_conv}
\end{eqnarray}
where $B^n_1,\ldots,B^n_M$ and $B_1,\ldots,B_M$ are the cells of $Q_n$
and $Q$, respectively, and $B^n_i \bigtriangleup B_i\coloneqq (B_i^n\setminus
B_i)\cup (B_i\setminus B_i^n)$. Since $Q$ has convex cells, the
boundary $\partial B_i$ of each cell $B_i$ has zero Lebesgue measure,
so $P(\partial B_i)=0$ because $P$ has a density. Since
$\partial(B_i\times \{j\})= \partial B_i \times \{j\}$, and $PQ(A \times \{j\})
= P(A \cap B_j)$,
 we have
\[PQ(\partial(B_i\times \{j\}))= P(\partial B_i \cap B_j)=0,\] for all
$i$ and $j$. Thus if $PQ^n \to PQ$ weakly, then
$PQ^n(B_i\times \{j\})
\to PQ(B_i\times \{j\})$  by the Portmanteau theorem, which is equivalent to
\[
P(B_i \cap B_j^n )\to P(B_i\cap B_j)
\]
for all $i$ and $j$. Since  $\{B^n_1,\ldots,B^n_M\}$  and $\{B_1,\ldots,B_M\}$
are both partitions of $\R^d$, this implies  $P(B^n_i
\bigtriangleup B_i)\to 0$ for all $i$, which in turns proves that $PQ^n\to PQ$ in
total variation via~\eqref{diff_conv}.

If $P$ has a positive density  and $P'$ admits a density, then $P'$ is
absolutely continuous with respect to $P$ and   so  $P(B^n_i
\bigtriangleup B_i)\to 0$ implies  $P'(B^n_i\bigtriangleup B_i)\to
0$. Combined with the preceding argument this proves the second statement in
part~(b).

(c) For any $A\in \B(\sX\times \sM)$ let
$A(x)\coloneqq \{y: (x,y) \in A\}$. Then
\begin{eqnarray*}
\lefteqn{ |P_n'Q_n(A)-P'Q_n(A)|} \\
& =&\left| \int_{\R^d}   Q_n(A(x)|x) P'_n(dx) -
 \int_{\R^d}   Q_n(A(x)|x) P'(dx)\right|  \nonumber \\
&\le &d_{TV}(P'_n,P'),
\end{eqnarray*}
where the inequality is due to \eqref{TValternative}. Taking the  supremum over all
$A$ yields
\[
d_{TV}(P'_nQ_n,P'Q_n) \le d_{TV}(P'_n,P').
\]

Hence
 \begin{eqnarray*}
\lefteqn{ d_{TV}(P'_nQ_n, P'Q) }\quad \\*
  &\le &  d_{TV}(P'_nQ_n, P'Q_n) + d_{TV}(P'Q_n, P'Q)\nonumber \\
 &\le&   d_{TV}(P'_n, P) + d_{TV}(P'Q_n, P'Q).
\end{eqnarray*}
From part~(b) we know that $Q_n\to Q$ in total variation at $P'$. Since $P'_n\to
P$ in total variation, we obtain $d_{TV}(P'_nQ_n, P'Q)\to 0$.
\qed

Recall from Definition~\ref{defsetS} in Section~\ref{SectionSpaceQuant} the
set $\S\subset \P(\R^d)$ of probability measures admitting
densities that are uniformly bounded and uniformly Lipschitz (with constants
determined by the conditional density $\phi(\,\cdot\,|x)$ of
$x_{t+1}=f(x_t,w_t)$ given $x_t=x$).  In Lemma~\ref{technicalLemma} we showed
that $\S$ contains all reachable states, i.e., $\pi_t \in \S$ for all $t\ge 1$
with probability 1 under any policy $\Pi\in \Pi_W$.

Lemma~\ref{keyTechnicalLemma}(a) immediately implies that for any sequence
$\{\mu_n\}$ in $\S$ and $\mu\in \S$,  $\mu_n\to \mu$ weakly if and only if
$\mu_n\to \mu$ in total variation. In this case  we simply say that
$\{\mu_n\}$ converges to $\mu$ in $\S$.

As discussed in Section~\ref{SectionSpaceQuant}, we can define the (quotient)
topology on $\Q_c$ induced by weak convergence of sequences at a given $P$
admitting a positive density. Lemma~\ref{keyTechnicalLemma}(b) implies that any
sequence in $\Q_c$ converging in this topology will converge both weakly and in
total variation at any $P'$ admitting a density.  In the rest of this section,
to say that $\{Q_n\}$ converges in $\Q_c$ will mean \emph{convergence in this
  topology}. We equip $S\times\Q_c$ with the corresponding product topology, and
continuity of any $F:\S\times \Q_c\to \R$ \emph{will be meant in this sense},
unless specifically stated otherwise.

\begin{lemma} \label{lemma_sqconv}
\begin{itemize}
\item[\em (a)] $\S$ is closed in $\P(\R^d)$.
\item[\em (b)] $\Q_c$ is compact.
\item[\em (c)] If $\{(\mu_n,Q_n)\}$ converges in
  $\S\times \Q_c$ to $(\mu,Q)\in \S\times \Q_c$ then  $\mu_nQ_n\to \mu
  Q$ in total variation. Thus any  $F:\S\times \Q_c\to \R$ is continuous if $F(\mu_n,Q_n)\to
    F(\mu,Q)$ whenever $\mu_nQ_n\to \mu  Q$ in total variation.

\end{itemize}

\end{lemma}

\begin{proof} (a) Recall that $\S$ is a uniformly bounded and uniformly
  equicontinuous family of densities.  Lemma~\ref{keyTechnicalLemma}(a) shows
  that if $\{\mu_n\}$ is a sequence in $\S$ and $\mu_n\to \mu$ weakly, then
  $\mu$ has a density. The proof also shows that some subsequence of (the
  densities of) $\{\mu_n\}$ converges to (the density of) $\mu$
  \emph{pointwise}. Thus $\mu$ must admit the same uniform upper bound and
  Lipschitz constant as all densities in $\S$, proving that $\mu\in \S$.

(b) The compactness of $\Q_c$  was shown in \cite[Thm. 5.8]{YukLinSIAM2010}.

(c)   If $\{(\mu_n,Q_n)\}$ converges in
  $\S\times \Q_c$ to $(\mu,Q)\in \S\times \Q_c$ then $\mu_n\to \mu$ in total
  variation. Since $\mu$ has a density, $Q_n\to Q$ in $\Q_c$ implies that
  $Q_n\to Q$ in total variation at $\mu$. Thus  $\mu_nQ_n\to \mu  Q$ in total
  variation by   Lemma~\ref{keyTechnicalLemma}(c).
\end{proof}

\subsection{Proof of Theorem  \ref{MeasurableSelectionApplies}} \label{ProofOptQuantizerSectionB}

The  first
statement of the following
theorem immediately implies
Theorem~\ref{MeasurableSelectionApplies}.

\begin{theorem}\label{thmexistencehypothesis}
 For  $t=T-1,\ldots,0$ define the
 value function $J^T_{t}$ at time $t$ recursively by
\[
J^T_{t}(\pi) =
 \inf_{Q \in \mathcal{Q}_c} \biggl( \frac{1}{T}c(\pi,Q) + E[J^T_{t+1}(\pi_{t+1})|
   \pi_t=\pi,Q_t=Q ]  \biggr)
\]
with $J^T_T\coloneqq0$ and $c(\pi,Q)$ defined in \eqref{eq_cdef}. Then for any
$t\ge 1$ and $\pi\in \mathcal{S}$ or $t=0$ and $\pi\in \S\cup \{\pi_0\}$,  the
infimum is achieved by some $Q$ in $\mathcal{Q}_c$.  Moreover, $J^T_t(\pi)$ is
continuous on $\mathcal{S}$.
\end{theorem}

The rest of this section is devoted to proving
Theorem~\ref{thmexistencehypothesis}. The proof is through backward induction in
$t$ combined with a series of lemmas
that show the continuity of both $c(\pi,Q)$ and $E[J^T_{t+1}(\pi_{t+1})|
\pi_t=\pi,Q_t=Q ]$ in $(\pi,Q)$.

\begin{lemma}\label{continuityofc} $c(\pi,Q)$ is continuous on $\mathcal{S}\times
\mathcal{Q}_c$.
\end{lemma}

\begin{proof}
If $\{(\pi_n,Q_n)\}$  converges in $\mathcal{S}\times
\mathcal{Q}_c$  then
$\pi_nQ_n\to \pi Q$ in total variation by Lemma~\ref{lemma_sqconv}(c).
We have to show that in this case
\begin{eqnarray*}
c(\pi_n,Q_n) &=& \inf_{\gamma} \int_{\R^d} \pi_n(dx)\sum_{i=1}^M Q_n(i|x) c_0(x,
  \gamma(i))  \\* 
&&  \to
\inf_{\gamma} \int_{\R^d} \pi(dx)\sum_{i=1}^M Q(i|x)  c_0(x, \gamma(i))\\
& =& c(\pi,Q).
\end{eqnarray*}
This follows verbatim from the proof of
\cite[Thm. 3.4]{YukLinSIAM2010}  where for any bounded  $c_0$
the convergence for a  fixed $\pi$ and $Q_n\to Q$
was shown.
\end{proof}

We now start proving Theorem~\ref{thmexistencehypothesis}. At $t=T-1$ we have
\[
J^T_{T-1}(\pi) = \inf_{Q\in \Q_c} c(\pi,Q).
\]
By Lemma~\ref{continuityofc} and the compactness of the set of
quantizers  $\mathcal{Q}_c$ (Lemma~\ref{lemma_sqconv}(b)) there exists an
optimal quantizer that achieves the infimum. The following lemma will
be useful.

\begin{lemma}\label{ContinuityofOptimalValue}
If $F: {\cal S} \times {\cal Q}_c \to \mathbb{R}$ is continuous then
$\inf_{Q\in \mathcal{Q}_c} F(\pi,Q)$ is achieved by some $Q$ in
$\mathcal{Q}_c $ and $\min\limits_{Q} F(\pi,Q)$
is continuous in $\pi$ on $\S$.
\end{lemma}

\emph{Proof.}  The existence of an optimal $Q$ in $\mathcal{Q}_c$ achieving
$\inf_{Q\in \mathcal{Q}_c} F(\pi,Q)$ is a consequence of the continuity of $F$
and the compactness of $\mathcal{Q}_c$. Assume  $\pi_n \to \pi$ in $\S$   and let   $Q_n$ be
optimal for $\pi_n$ and  $Q$ optimal for $\pi$. Then
\begin{eqnarray*}
\lefteqn{ \big|\min_{Q'} F(\pi_n,Q') - \min_{Q'} F(\pi,Q')\big| }
\nonumber \\* 
& & \hspace{-10pt} \leq \max\! \bigg(\! F(\pi_n,Q) - F(\pi,Q), F(\pi,Q_n) - F(\pi_n,Q_n)
\bigg). 
\end{eqnarray*}
The first term in the maximum converges to zero since $F$ is continuous. To show
that the second converges to zero, suppose to the contrary that for some
$\epsilon > 0$ and for a subsequence $\{(\pi_{n_k},Q_{n_k})\}$,
\begin{equation}
\label{eq:contr}
|F(\pi,Q_{n_k}) - F(\pi_{n_k},Q_{n_k})| \geq \epsilon.
\end{equation}
By  Lemma~\ref{lemma_sqconv}(a), there is a further subsequence $\{n'_k\}$ of
$\{n_k\}$ such that 
$\{Q_{n'_k}\}$  converges to some $Q'$ in $\Q_c$. Then $\{(\pi, Q_{n'_k})\}$ and
$\{(\pi_{n'_k}, Q_{n'_k})\}$  both converge to $(\pi,Q')$,
which contradicts \eqref{eq:contr} since $F$ is continuous. \qed

As a consequence of Lemmas~\ref{continuityofc} and
\ref{ContinuityofOptimalValue}, $J^T_{T-1}(\pi)$ is continuous on
$\S$, proving Theorem~\ref{thmexistencehypothesis} for $t=T-1$. 
To prove the theorem
for all $t=T-2,\ldots,0$, we apply backward induction. Assume that the both
statements of the theorem hold for $t'=T-1,\ldots,t+1$.
We want to show  that the
minimization problem
\begin{eqnarray}
\lefteqn{J^T_t(\pi) } \nonumber \\
 & & \!\! \!\!\!\!\!\!\!\! = \min_{Q \in \mathcal{Q}_c}
\!\!\bigg(  \frac{1}{T}  c(\pi,Q)\! +  \!
E\bigl[J^T_{t+1}(\pi_{t+1})|\pi_t\!=\!\pi,Q_t\!=\!Q\bigr]\!\!
\bigg) \label{recursionDP2} 
\end{eqnarray}
has a  solution and $J^T_{t}(\pi)$ is continuous on $\S$.

Consider  the conditional probability distributions given by
\begin{eqnarray}
\lefteqn{\!\!\!\! \hat{\pi}(m,\pi,Q)(C) 
 \coloneqq  P(x_{t+1} \in C | \pi_t=\pi,Q_t=Q, q_t = m) }\nonumber \\
& & \!\!\!\!\! \!\!\!\!\!\!\!\!   = \frac{1}{\pi(Q^{-1}(m))} \int_C  \bigg(\int_{\R^d}\pi(dx) 1_{\{x \in B_m\}}
\phi(z|x)  \bigg)\,  dz \label{updateQuantizationNext} 
\end{eqnarray}
(if $\pi(Q^{-1}(m))=0$, then $\hat{\pi}(m,\pi,Q)$ is set arbitrarily). Note that
\begin{eqnarray}
\lefteqn{ E\bigl[J^T_{t+1}(\pi_{t+1}) |\pi_t=\pi,Q_t=Q\bigr]  }\qquad  \nonumber \\*
&= &  \sum_{m=1}^M   J^T_{t+1}\bigl(\hat{\pi}(m,\pi,Q)\bigr)
\pi\bigl(Q^{-1}(m)\bigr), \label{eq_totalexp}
\end{eqnarray}
where
\[
 \pi\bigl(Q^{-1}(m)\bigr)=   P(q_{t}=m | \pi_{t}=\pi,Q_{t}=Q).
\]

The following lemma will imply that if $(\pi_n, Q_n)\to ( \pi, Q)$ in $\S\times \Q_c$, then
\begin{eqnarray}
\lefteqn{  J^T_{t+1}\bigl(\hat{\pi}(m,\pi_n,Q_n)\bigr)
  \pi_n\bigl(Q_n^{-1}(m)\bigr)} \nonumber \qquad \qquad  \\
&&    \to   J^T_{t+1}\bigl(\hat{\pi}(m,\pi,Q)\bigr) \pi\bigl(Q^{-1}(m)\bigr)
 \label{eq_sumconv}
\end{eqnarray}
for all $m$.

\begin{lemma} \label{TotalVarConvBins}
If $\pi_n Q_n \to
\pi Q$ in total variation, then   $\hat{\pi}(m,\pi_n,Q_n)\to
\hat{\pi}(m,\pi,Q))$ in total variation for every $m=1,\ldots,M$ with $\pi(Q^{-1}(m)) >
0$.

\end{lemma}

\emph{Proof.} Let $B_1,\ldots,B_M$ and $B_1^n,\ldots,B_M^n$ denote the cells of
$Q$ and $Q_n$, respectively. Since for any Borel set $A$, $\pi_n Q_n(A\times
\{j\}) =\pi_n(A\cap B^n_j)$, the convergence of $\pi_n Q_n$ to $\pi Q$ implies
that $\pi_n(A\cap B^n_m)\to \pi(A\cap B_m)$. This implies $\pi_n(B_i \cap
B_j^n)\to \pi(B_i\cap B_j)$ for all $i$ and $j$, from which we obtain for all
$m=1,\ldots,M$, 
\begin{equation}
\label{eq_pnmbconv}
\pi_n(B_m^n)\to \pi(B_m), \quad \pi_n(B_m^n\bigtriangleup B_m)\to
0, 
\end{equation}
where $B_m^n \bigtriangleup B_m = (B_m^n \setminus B_m)\cup (B_m\setminus B_m^n)$.

If $\pi(B_m)>0$, the probability distribution
$\hat{\pi}(m,\pi,Q)$ has density
\[
\hat{\pi}(m,\pi,Q)(z)= \frac{1}{\pi(B_m)}  \int_{B_m} \pi(dx)
 \phi(z|x)
\]
so by Scheffe's theorem \cite{Bil86} it suffices to show that
$\hat{\pi}(m,\pi_n,Q_n)(z) \to \hat{\pi}(m,\pi,Q)(z)$  for all $z$. As $\pi(B_m)>0$ by assumption and $\pi_n(B_m^n)\to \pi(B_m)$, it is enough to
establish the convergence of $ v_n^m(z)\coloneqq \int_{B_m^n} \pi_n(dx) \phi(z |x) $ to  $ v^m(z)\coloneqq \int_{B_m} \pi(dx)
\phi(z|x) $.

For any $z\in \R^d$ we have
\begin{eqnarray}
\lefteqn{  |v_n^m(z)-v^m(z)| } \nonumber \\*
&\le & \left| \int_{\R^d} \pi_n(dx) \bigl (1_{\{x\in
     B_m^n\}} - 1_{\{x\in B_m\}} \bigr)  \phi(z|x) \right|
  \nonumber \\
 & & \mbox{} +   \left| \int_{\R^d} 1_{\{x\in
    B_m\}}\bigl(\pi_n(x)-\pi(x)\bigr)  \phi(z|x)\, dx \right|
  \nonumber\\
 &\le & \int_{B_m^n\bigtriangleup B_m} \!\!\!\!\!\! \pi_n(dx)  \phi(z|x)
   +   \int_{\R^d}\bigl|\pi_n(x)-\pi(x)\bigr|
  \phi(z|x) \, dx  \nonumber\\
 &\le & C \biggl[ \pi_n(B_m^n\bigtriangleup B_m) +
    d_{TV}(\pi_n,\pi)\biggl],
\end{eqnarray}
where $C$ is a uniform upper bound on $\phi$. Since both terms in the
brackets converge to zero as $n\to \infty$, the proof is complete.  \qed

Now if $(\pi_n, Q_n)\to ( \pi, Q)$ in $\S\times \Q_c$, then by
Lemma~\ref{lemma_sqconv} and \eqref{eq_pnmbconv} we have $
\pi_n\bigl(Q_n^{-1}(m)\bigr)\to \pi\bigl(Q^{-1}(m)\bigr)$ for all $m$. If
$\hat{\pi}(m,\pi,Q) >0$ for some $m$, then Lemma~\ref{TotalVarConvBins} implies
that $\hat{\pi}(m,\pi_n,Q_n)\to \hat{\pi}(m,\pi,Q)$ in total variation; hence
\eqref{eq_sumconv} holds in this case by the continuity of $J^T_{t+1}$. If $
\pi\bigl(Q^{-1}(m)\bigr)=0$, then by \eqref{eq_pnmbconv} and the boundedness of
the cost \eqref{eq_sumconv} holds again. In view of \eqref{eq_totalexp}, we
obtain that $E\bigl[J^T_{t+1}(\pi_{t+1}) |\pi_{t}=\pi,Q_{t}=Q\bigr]$ is
continuous on $\S\times \Q_c$.

We have shown that both expressions on the right side of (\ref{recursionDP2})
are continuous on $\S\times \Q_c$. By Lemma~\ref{ContinuityofOptimalValue} the
minimization problem \eqref{recursionDP2} has a solution and $J^T_{t}(\pi)$ is
continuous on $\S$, proving the induction hypotesis for $t'=t$. 

To finish the
proof we have to consider the last step $t=0$ separately. For $t=0$ we have that
if $\pi$ admits a positive  density, then there exists a minimizing $Q$ for
  \[
J^T_{0}(\pi) =
 \inf_{Q \in \mathcal{Q}_c} \biggl(
 \frac{1}{T}c(\pi_0,Q) + E[J^T_{1}(\pi_{1})|
   \pi_0=\pi,Q_t=Q ]   \biggl)
\]
by Lemma~\ref{ContinuityofOptimalValue} since the preceding proofs readily imply
that both $c(\pi,Q)$ and $E[J^T_{1}(\pi_{1})| \pi_0=\pi,Q_t=Q ]$ are continuous in
$Q$ as long as $\pi$ admits a positive density. If $\pi$ is a point mass on $x_0$,
then any $Q$ is optimal. This establishes
Theorem~\ref{thmexistencehypothesis}. \qed

\subsection{Proof of Theorem
  \ref{MeasurableSelectionApplies2}}\label{MeasurableApplies2Proof}

The first statement of the following counterpart of
Theorem~\ref{thmexistencehypothesis}
immediately implies  Theorem~\ref{MeasurableSelectionApplies2}.

\begin{theorem}\label{thmexistencehypothesis1}
  Consider Assumption~\ref{AssumptionA'}.  For $t=T-1,\ldots,0$ define
  the value function $J^T_{t}$ at time $t$ recursively by
\[
J^T_{t}(\pi) =
 \inf_{Q \in \mathcal{Q}_c} \biggl( \frac{1}{T}c(\pi,Q) + E\bigl[J^T_{t+1}(\pi_{t+1})|
   \pi_t=\pi,Q_t=Q \bigr]  \biggr)
\]
with $J^T_T\coloneqq0$ and $c(\pi,Q)$ defined in \eqref{eq_cdef}. Then for any
$t\ge 1$ and $\pi\in \mathcal{S}$ or $t=0$ and $\pi\in \S\cup \{\pi_0\}$ the
infimum is achieved by some $Q$ in $\mathcal{Q}_c$.

Moreover, $J^T_t(\pi)$ is continuous on $\mathcal{S}$ in the sense that if
$\pi_n\to \pi$ and $\{\pi_n\}$ satisfies the uniform integrability condition
\begin{equation}
 \label{eq_unint}
\lim_{L\to \infty} \sup_{n\ge 1} \int_{\{\|x\|^2\ge L\}} \|x\|^2
\pi_n(dx) =0,
\end{equation}
then  $J^T_t(\pi_n)\to J^T_t(\pi)$.
\end{theorem}

To prove Theorem~\ref{thmexistencehypothesis1} we need to modify the
proof of Theorem~\ref{thmexistencehypothesis} only in view of the
unboundedness of the cost, which affects the proof of the continuity
of $c(\pi,Q)$ and $ E\bigl[J^T_{t+1}(\pi_{t+1})| \pi_t=\pi,Q_t=Q \bigr]$.

We first establish the continuity of $c(\pi, Q)$ in a more restricted
sense than in Lemma~\ref{continuityofc}.  We know from \eqref{eq_gammaopt}
that given $\pi_t=\pi$ and $Q_t=Q$ with cells $B_1,\ldots,B_M$,
the unique optimal
receiver policy is given, for any $m$ such that $\pi(B_m)>0$, by
\[
\gamma(m)=\int_{B_m} x \pi(dx).
\]
If $\pi(B_m)=0$, then $\gamma(m)$ is arbitrary.
Using this optimal receiver policy, define
$\bar{Q}: \R^d\to \R^d$ by
\[
\bar{Q}(x)= \gamma(Q(x)).
\]
Note that $c(\pi,Q)= \int  \big\|x-\bar{Q}(x)\big\|^2  \pi(dx) $ and that for all $m$,
\begin{eqnarray}
\int  \big\|x-\bar{Q}(x)\big\|^2 1_{\{x\in B_m\}} \pi(dx) &= &
\int_{B_m}\|x-\gamma(m)\|^2 \pi(dx) \nonumber \\
&\le & 
\int_{B_m}\|x\|^2\pi(dx) \label{eq_secmombound} 
\end{eqnarray} 
which implies
\begin{equation}
\label{eq_secmombound1}
c(\pi,Q)\le  \int \|x\|^2\pi(dx).
\end{equation}

\begin{lemma}
\label{lem_cconv}
Assume  $(\pi_n,Q_n)\to (\pi,Q)$ in $\mathcal{S} \times
 \mathcal{Q}_c$ and  $\{\pi_n\}$
satisfies the  uniform integrability condition
\eqref{eq_unint}.
Then $c(\pi_n,Q_n)\to c(\pi,Q)$.
\end{lemma}

\emph{Proof.} \
If $B^n_1,\ldots,B^n_M$ denote the  cells of $Q_n$ and let
$B_1,\ldots,B_M$ be the cells of $Q$. By \eqref{eq_pnmbconv}
we have $\pi_n(B_m^n)\to \pi(B_m)$ and $\pi_n(B_m^n\bigtriangleup
B_m)\to 0$.  Let $I=\bigl\{m\in \{1,\ldots,M\}: \pi(B_m) >0\bigr\}$.
We have for  any $L>0$ and $m\in I$,
\[
\int_{B_m^n} x  1_{\{\|x\|^2< L\}} \pi_n(dx) \to  \int_{B_m} x  1_{\{\|x\|^2<
  L\}} \pi(dx).
\]
This and  a standard truncation argument that makes  use of
\eqref{eq_unint} imply
\[
\int_{B_m^n} x  \pi_n(dx) \to  \int_{B_m} x   \pi(dx),\quad m\in I
\]
so the optimal receiver policy $\gamma_n$ for $Q_n$ satisfies
$\gamma_n(m)\to \gamma(m)$ for all $m\in I$. In particular, this
implies that for all $m\in I$,
\[
D_m\coloneqq   \sup_{n\ge 1} \sup_{x\in  B_m^n} \|\bar{Q}_n(x)\|^2 =\sup_{n\ge 1}
\gamma_n(m)   <\infty.
\]
In turn,  the parallelogram law gives for $m\in I$ and  $x\in B_m^n$,
\begin{equation}
  \label{eq_unibound}
\|x-\bar{Q}_n(x)\|^2 \le 2\|x\|^2 + 2\| \bar{Q}_n(x)\|^2 \le 2\|x\|^2 +2D_m
\end{equation}
so for $m\in I$ we obtain
\begin{eqnarray}
\lefteqn{ \lim_{L \to \infty} \sup_{n\ge 1} \int_{B_m^n} \pi_n(dx)
  \big\|x-\bar{Q}_n(x)\big\|^2 1_{\{ \|x\|^2  \geq L \}}
}\nonumber \\
& \le  &    \lim_{L \to \infty} \sup_{n\ge 1} \int_{\R^d}
\pi_n(dx)   \bigl( 2\|x\|^2+ 2D_m\bigr)  1_{\{ \|x\|^2\ge L  \}} \nonumber \\
&=& 0 ,\label{eq_unint1}
\end{eqnarray}
where the second  limit is zero due to \eqref{eq_unint}.

Since
$\pi_n\to \pi$ in total variation,  $\pi_n(B_m^n)\to \pi(B_m)$ and
$\pi_n(B_m^n\bigtriangleup
B_m)\to 0$, and since   $\|x-\bar{Q}_n(x)\|^2$ is
uniformly bounded if $\|x\|^2 < L$ by \eqref{eq_unibound}, we have
\begin{eqnarray*}
\lefteqn{
 \int_{B_m^n} \pi_n(dx)
  \big\|x-\bar{Q}_n(x)\big\|^2 1_{\{ \|x\|^2  < L \}} } \qquad \\
&\to &  \int_{B_m} \pi(dx)
  \big\|x-\bar{Q}(x)\big\|^2 1_{\{ \|x\|^2  < L \}} .
\end{eqnarray*}
Then uniform integrability \eqref{eq_unint1} and a standard
truncation argument yield for  $m\in I$
\begin{equation}
\label{eq_bmconv}
 \int_{B_m^n} \pi_n(dx)
  \big\|x-\bar{Q}_n(x)\big\|^2 \to
 \int_{B_m} \pi(dx)
  \big\|x-\bar{Q}(x)\big\|^2.
\end{equation}

Assume  $m\notin I$. Then we have
\[
 \int_{B_m^n}\|x\|^2 \pi_n(dx)
\le  \int_{\{\|x\|^2\ge L\}   }\|x\|^2 \pi_n(dx)  + L\pi_n(B_m^n) \to 0
\]
from  \eqref{eq_unint} and since  $\pi_n(B_m^n)\to 0$. In view of
\eqref{eq_secmombound} we obtain
\[
 \int_{B_m^n} \pi_n(dx)
  \big\|x-\bar{Q}_n(x)\big\|^2 \to 0.
\]
This and \eqref{eq_bmconv} give
\begin{eqnarray*}
  c(\pi_n,Q_n) &= &  \int_{\R^d}\pi_n(dx)
  \big\|x-\bar{Q}_n(x)\big\|^2  \\
&  &  \to 
 \int_{\R^d} \pi(dx)
  \big\|x-\bar{Q}(x)\big\|^2 = c(\pi,Q)
\end{eqnarray*}
which proves the lemma. \qed

The following variant of Lemma~\ref{ContinuityofOptimalValue} lemma will
be useful.

\begin{lemma}\label{ContinuityofOptimalValue1}
  Assume $F: \S \times \Q_c \to \mathbb{R}$ is continuous in the sense that if
$(\pi_n,Q_n) \to (\pi,Q)$ in $\mathcal{S} \times \mathcal{Q}_c$ and $\{\pi_n\}$
satisfies the uniform integrability condition \eqref{eq_unint}, then
$F(\pi_n,Q_n) \to F(\pi,Q)$.
  Then $\inf_{Q\in \mathcal{Q}_c} F(\pi,Q)$ is achieved by some $Q$ in
  $\mathcal{Q}_c $ and $\min_Q F(\pi,Q)$ is continuous in $\pi$ in the
  sense that if $\pi_n\to \pi$ in $\S$ and $\{\pi_n\}$ is uniformly integrable,
  then  $\min_Q F(\pi_n,Q)\to \min_Q F(\pi,Q)$.
\end{lemma}

\emph{Proof.} \ The existence of an optimal $Q$ for any $\pi \in \mathcal{S}$ is
a consequence of the compactness of $\mathcal{Q}_c$.  The rest of the proof
follows verbatim the proof of Lemma~\ref{ContinuityofOptimalValue} with the
convergence sequence $\{\pi_n\}$ also assumed to be uniformly integrable. \qed

\medskip

Lemmas \ref{lem_cconv} and \ref{ContinuityofOptimalValue1} prove
Theorem~\ref{thmexistencehypothesis1} for $t=T-1$. To prove the theorem
for all $t$, we apply backward induction. Assume that the both
statements of the theorem hold for $t'=T-1,\ldots,t+1$.

Recall the conditional distribution  $\hat{\pi}(m,\pi,Q)$ defined in
\eqref{updateQuantizationNext}. The following lemma shows that the uniform
integrability condition is inherited in the induction step.

\begin{lemma}\label{UnifIntegrableAllSamplePaths}
  Assume $(\pi_n,Q_n) \to (\pi,Q)$ in $\mathcal{S} \times \mathcal{Q}_c$ and
  $\{\pi_n\}$ satisfies the uniform integrability condition \eqref{eq_unint}. If
  cell $B_m$ of $Q$ satisfies $\pi(B_m)>0$, then
  $\{ \hat{\pi}(m,\pi_n,Q_n)\}$ is uniformly integrable in the sense of \eqref{eq_unint}.
\end{lemma}

\emph{Proof.} \ Let $B_m^n$ denote the $m$th cell of $Q_n$.  Since $\pi(B_m)>0$,
we have $\pi_n(B_m^n)>0$ for $n$ large enough, so
\begin{eqnarray*}
  \lefteqn{ \int_{\R^d} \|z\|^2 \hat{\pi}(m,\pi_n,Q_n)(z)1_{\{\|z\|^2\ge L\}} \,
    dz = }
  \\*
& & \hspace{-20pt} \frac{1}{\pi(B_m^n)} \!\int\limits_{\R^d} \!
\int\limits_{\R^d} \! \|f(x,w)\|^2
1_{\{\|f(x,w)\|^2\ge 
    L\}}  1_{\{x\in B_m^n\}}\pi_n(dx)\nu_w(dw).
\end{eqnarray*}
Since $\|f(x,w)\|^2\le 2K^2 \big(\|x\|^2 +\|w\|^2\big)$ by Assumption~\ref{AssumptionA'}(i), we have
\[
  1_{ \{\|f(x,w)\|^2\ge L  \}}\le    1_{ \{\|x\|^2\ge L/(4K^2)  \}}+
  1_{ \{\|w\|^2\ge L/(4 K^2) \}}
\]
and so
\begin{eqnarray*}
  \lefteqn{ \int_{\R^d} \|z\|^2 \hat{\pi}(m,\pi_n,Q_n)(z)1_{\{\|z\|^2\ge L\}} \, dz}
  \nonumber \\*
  & & \hspace{-30pt} \le \frac{1}{\pi(B_m^n)} \int_{\R^d}  \int_{\R^d}  2K^2\bigl( \|x\|^2 +
  \|w\|^2\bigr)  \times  \\
  & & \times \bigl(    1_{ \{\|x\|^2\ge L/(4K^2)  \}}+
  1_{ \{\|w\|^2\ge L/(4 K^2) \}} \bigr) \pi_n(dx)\nu_w(dw) \nonumber \\
  && \hspace{-30pt} =  \frac{2K^2}{\pi(B_m^n)} \int_{\R^d} \|x\|^2 1_{ \{\|x\|^2\ge L/(4K^2)  \}}
  \pi_n(dx)    \\
  & & \hspace{-20pt}  \mbox{} +
  \frac{2K^2}{\pi(B_m^n)} \left( \int_{\R^d}\!  \|x\|^2 \pi_n(dx)\right)
  \left( \int_{\R^d} \!  1_{ \{\|w\|^2\ge L/(4K^2)  \}}\nu_w(dw) \right)
  \\
  & & \hspace{-20pt}  \mbox{} +
  \frac{2K^2}{\pi(B_m^n)} \left( \int_{\R^d}\! \|w\|^2 \nu_w(dw) \right)
  \left( \int_{\R^d}  \! 1_{ \{\|x\|^2\ge L/(4K^2)  \}} \pi_n(dx)\right)
  \\
  && \hspace{-20pt}   \mbox{} +  \frac{2K^2}{\pi(B_m^n)} \int_{\R^d} \!\|w\|^2
  1_{  \{\|w\|^2\ge L/(4K^2)  \}} 
  \nu_w(dw).  
\end{eqnarray*}
Recall that $\pi_n(B_m^n)\to \pi(B_m)$.  Thus the first term in sum above
converges to zero as $L\to \infty$ uniformly in $n$ by \eqref{eq_unint}. The
uniform convergence to zero of the other three terms in the sum follows since
$\int \|w\|^2\nu_w(dw)<\infty$ and $\sup_{n\ge 1} 
\int \|x\|^2\pi_n(x)\, dx<\infty$ by \eqref{eq_unint}.  This proves
\[
\lim_{L\to \infty}\sup_{n\ge 1}  \int_{\R^d} \|z\|^2
\hat{\pi}(m,\pi_n,Q_n)(z)1_{\{\|z\|^2\ge L\}} \, dz =0
\]
as claimed. \qed

\medskip

The next lemma shows the continuity of $
E\bigl[J^T_{t+1}(\pi_{t+1}) |\pi_t=\pi,Q_t=Q\bigr]$.

\begin{lemma}
\label{lem_condentcont}
$E\bigl[J^T_{t+1}(\pi_{t+1}) |\pi_t=\pi,Q_t=Q\bigr]$ is continuous on $\S\times
\Q_c$ in the sense of Lemma~\ref{ContinuityofOptimalValue1}. 
\end{lemma}

\emph{Proof.}\ Assume $(\pi_n,Q_n) \to (\pi,Q)$ in $\mathcal{S} \times
\mathcal{Q}_c$ and $\{\pi_n\}$ satisfies the uniform integrability condition
\eqref{eq_unint}. Let $B_1,\ldots,B_M$ and $B_1^n,\ldots,B_M^n$ denote the cells
of $Q$ and $Q_n$, respectively.  In view of \eqref{eq_totalexp} and the fact
that $\pi_n(B_m^n)\to \pi(B_m)$, we need to prove that for all $m$ with
$\pi(B_m)>0$,
\begin{equation}
\label{eq_normalconv}
  J^T_{t+1}(\hat{\pi}(m,\pi_n,Q_n))\to   J^T_{t+1}(\hat{\pi}(m,\pi,Q))
\end{equation}
and for $m$ with $\pi(B_m)=0$,
\begin{equation}
\label{eq_zeroconv}
    J^T_{t+1}(\hat{\pi}(m,\pi_n,Q_n))\pi_n(B_m^n)\to 0.
\end{equation}

The convergence in \eqref{eq_normalconv} follows from
Lemmas~\ref{TotalVarConvBins} and \ref{UnifIntegrableAllSamplePaths}, and the induction
hypothesis that $J^T_{t+1}(\,\cdot\,)$ is continuous  along convergent and  uniformly
integrable sequences  in $\S$.

To prove \eqref{eq_zeroconv} first note that from
\eqref{eq_secmombound1} we have
\[
J^T_{t+1}(\pi_{t+1})\le E\biggl[ \frac{1}{T} \sum_{i=t+1}^{T-1}
\|x_{i}\|^2  \biggr],
\]
where $x_{t+1}$ has distribution $\pi_{t+1}$ and $x_{i}=f(x_{i-1},w_{i-1})$,
where $w_{t+1},\ldots,w_{T-1}$ are independent of $x_{t+1}$.  Accordingly,
\begin{eqnarray}
\lefteqn{ J^T_{t+1}(\hat{\pi}(m,\pi_n,Q_n))\pi_n(B_m^n)  }\nonumber\qquad  \\
& \le & 
 E\biggl[ \frac{1}{T} \biggl(\, \sum_{i=t+1}^{T-1}
\|x_{i,n}\|^2\biggr)1_{\{x_{t,n}\in B_m^n\}}  \biggr], \label{eq_xtnbound}
 \end{eqnarray}
where $x_{t,n}$ has distribution $ \pi_n$.

Now note that the assumption $\|f(x,w)\le K\big( \|x\|+\|w\|)$ and the inequality $
\|x+y\|^2 \le 2\|x\|^2 +2\|y\|^2$ imply the upper bound
\begin{equation}
  \label{eq_xtbound}
\|x_{t+j,n}\|^2 \le (2K^2)^j  \|x_{t,n}\|^2 + \sum_{i=0}^{j-1} (2K^2)^{j-i}
  \|w_{t+i}\|^2.
\end{equation}
Thus for any $j=1,\ldots,T-t-1$ we have
\begin{eqnarray*}
\lefteqn{  E\bigl[\|x_{t+j,n}\|^2 1_{\{x_{t,n}\in B_m^n\}}  \bigr] }
\nonumber  \\*
&\le & (2K^2)^j  E\biggl[  \|x_{t,n}\|^2 1_{\{x_{t,n}\in   B_m^n\}}  \nonumber
\\*
  & & \qquad \mbox{}   + \sum_{i=1}^j
  (2K^2)^{1-i}\|w_{t+i-1}\|^2 1_{\{x_{t,n}\in   B_m^n\}}     \biggr] \nonumber\\
&=&  (2K^2)^j  \biggl( E\bigl[ \bigl\|
   x_{t,n}\bigr\|^2  1_{\{x_{t,n}\in      B_m^n\}}
   \bigr]  \nonumber  \\
& & \qquad \qquad \mbox{}  + \sum_{i=1}^{j} (2K^2)^{1-i}E\bigl[
   \bigl\| w_{t+i-1} \bigr\|^2  \bigr] \pi_n(B_m^n)\biggr),  
\end{eqnarray*}
where we used the independence of $w_t,\ldots,w_{T-1}$ and $x_{t,n}$.  The
first expectation in  the last equation  converges to zero as $n\to \infty$ since
$\{\pi_n\}$ is uniformly integrable and $\pi_n(B_n^m)\to \pi(B_m)=0$, while the
second one converges to zero since $\pi_n(B_n^m)\to 0$.  This proves
that the right side of \eqref{eq_xtnbound} converges to zero, finishing the
proof of the lemma. \qed

\medskip

Lemmas~\ref{lem_cconv}  and \ref{lem_condentcont} show that
\[
F_t(\pi,Q)\coloneqq   \frac{1}{T}c(\pi,Q) + E\bigl[J^T_{t+1}(\pi_{t+1}) |\pi_t=\pi,Q_t=Q\bigr]
\]
satisfies the conditions of
Lemma~\ref{ContinuityofOptimalValue1}, which in turn proves
the induction hypothesis for $t'=t$. For the last step $t=0$ a similar argument
as in the proof of  Theorem~\ref{thmexistencehypothesis} applies (but here we
also need the condition $E_{\pi_0}\bigl[\|x_0\|^2\bigr]<\infty$).
 This finishes the proof of
Theorem~\ref{thmexistencehypothesis1}. \qed

\subsection{Proof of Theorem~\ref{thm_QWopt}} \label{Opt_PW_policy}

Define
\[
J_{\pi^*}(T) \coloneqq \inf_{\Pi \in \Pi_A} \inf_{\gamma}   E^{\Pi,\gamma}_{\pi^*}\biggl[\frac{1}{T}   \sum_{t=0}^{T-1}
c_0(x_t,u_t)\biggr]
\]
and note that $\limsup_{T\to \infty} J_{\pi^*}(T) \le J_{\pi^*}$. Thus there
exists an \emph{increasing} sequence of time indices $\{T_k\}$ such that for all $k=1,2,\ldots$,
\begin{equation}
  \label{eq_kopt}
 J_{\pi^*}(T_k) \le J_{\pi^*} +\frac{1}{k}.
\end{equation}
A key
observation is that  by Theorem~\ref{WalrandVaraiyaTheorem}
for all $k$ there exists $\Pi_k =\{\heta^{(k)}_t\} \in \Pi_W$ (a Markov policy) such that
\begin{equation}
  \label{eq_kopt1}
J_{\pi^*}(\Pi_k,T_k)\coloneqq E^{\Pi_k}_{\pi^*}\biggl[\frac{1}{T_k}   \sum_{t=0}^{T_k-1}
 c(\pi_t,Q_t)\biggr] \le J_{\pi^*}(T_k) + \frac{1}{k}.
\end{equation}

Now let  $n_1=1$ and for $k=2,3,\ldots$,
choose the positive integers $n_k$ inductively as
\begin{equation}
  \label{eq_nk}
n_k  =\left\lceil   k \cdot \max \biggl( \frac{T_{k+1}}{T_k}, \frac{n_{k-1}T_{k-1}}{T_k}
\biggr) \right\rceil,
\end{equation}
where $\lceil x\rceil$ denotes the smallest  integer greater than equal to $x$.
Note that the definition of $n_k$ implies $n_k T_k\ge
k n_{k-1}T_{k-1}$. Thus letting $T_k'= n_kT_k$ for all $k$
we have 
\begin{equation}
  \label{eq_tk}
T_k'\ge kT_{k-1}',
  \end{equation}
and hence
\begin{equation}
  \label{eq_kopt2}
\lim_{k\to \infty} \frac{\sum_{l=1}^k T_l'}{T_k'} =1.
\end{equation}
Now let $N_0= 0$, $N_k=\sum_{i=1}^k T_k'$ for $k\ge 1$, and define the policy
$\Pi= \{\heta_t\} \in \Pi_W$ by piecing together, in a periodic fashion, the
initial segments of $\Pi_k$ as follows:
\begin{itemize}
\item[(1)] For $t=N_{k-1}+j T_k$, where $k\ge 1$ and  $0\le j<n_k$,  let
  $\heta_t(\,\cdot\,)\equiv \heta^{(k)}_0(\pi^*)$;

\item[(2)] For $t=N_{k-1}+j  T_k+i$, where $k\ge 1$, $0\le j<n_k$, and
  $1\le i< T_k$,    let
  $\heta_t= \heta^{(k)}_i$.
\end{itemize}

In the rest of the proof we show that $\Pi$ is optimal. First note that by the
stationarity of $\{x_t\}$  we have, for all $k\ge 1$ and $j=0,\ldots,n_k-1$,
\[
  E^{\Pi}_{\pi^*}\biggl[\,\sum_{t=N_{k-1}+jT_k}^{N_{k-1}+(j+1)T_k-1}
  c(\pi_t,Q_t)\biggr] = T_k J_{\pi^*}(\Pi_k,T_k).
\]
Hence, for  $T=N_{k-1}+j  T_k+i$, where $k\ge 3$, $0\le j<n_k$, and
  $0\le i< T_k$,   we have
\begin{eqnarray}
\lefteqn{      E^{\Pi}_{\pi^*}\biggl[\frac{1}{T}\sum_{t=0}^{T-1}
  c(\pi_t,Q_t)\biggr]  }  \nonumber \quad \\*
&=&     E^{\Pi}_{\pi^*}\biggl[\frac{1}{T}\sum_{t=0}^{N_{k-2}-1}
c(\pi_t,Q_t)\biggr] +     E^{\Pi}_{\pi^*}\biggl[\frac{1}{T}\sum_{t=N_{k-2}}^{T-1}
  c(\pi_t,Q_t)\biggr] \nonumber \\
&=& \frac{1}{T} \sum_{l=1}^{k-2} T_l' J_{\pi^*}(\Pi_l,T_l)   \label{eq_JTbound} \\
& & \mbox{}  +  \frac{1}{T} \Bigl(
T_{k-1}'J_{\pi^*}(\Pi_{k-1},T_{k-1})+ jT_k
J_{\pi^*}(\Pi_k,T_k)\Bigr)  \label{eq_JTbound1} \\
& & \mbox{} + E^{\Pi}_{\pi^*}\biggl[\frac{1}{T}\sum_{t=N_{k-1}+jT_k}^{T-1}
  c(\pi_t,Q_t)\biggr]  \label{eq_JTbound2}
\end{eqnarray}
(the last sum is empty if $i=0$).

Let $\hat{C}$ be a uniform upper bound on the cost $c_0$. Since $T\ge N_{k-1}$,
\eqref{eq_JTbound} can be bounded as
\begin{eqnarray}
  \frac{1}{T} \sum_{l=1}^{k-2} T_l' J_{\pi^*}(\Pi_l,T_l)  &\le &
  \hat{C}  \frac{1}{N_{k-1}}\sum_{l=1}^{k-2} T_l' =   \hat{C}
  \frac{N_{k-2}}{N_{k-1}} \nonumber \\
  &=&  \hat{C}
  \frac{\frac{N_{k-2}}{T_{k-2}'}}{\frac{N_{k-2}}{T_{k-2}'} +
    \frac{T_{k-1}'}{T_{k-2}'}} \to 0  \label{eq_Tbound1}
\end{eqnarray}
as $k\to \infty$ since $\frac{N_{k-2}}{T_{k-2}'}\to 1$ from \eqref{eq_kopt2} and
$\frac{T_{k-1}'}{T_{k-2}'}\ge k-1 $ from \eqref{eq_tk}. 

Since $T_{k-1}' + jT_k \le T$,  \eqref{eq_JTbound1} can be
upper bounded as
\begin{eqnarray}
\lefteqn{  \frac{1}{T} \Bigl(
T_{k-1}'J_{\pi^*}(\Pi_{k-1},T_{k-1})+ jT_k J_{\pi^*}(\Pi_k,T_k)\Bigr) }
\nonumber \quad \\
&\le &
\max\Bigl(J_{\pi^*}(\Pi_{k-1},T_{k-1}),J_{\pi^*}(\Pi_k,T_k) \Bigr).    \label{eq_Tbound2}
\end{eqnarray}
Finally, the expectation in \eqref{eq_JTbound2} is upper bounded as
\begin{eqnarray}
 E^{\Pi}_{\pi^*}\biggl[\frac{1}{T}\sum_{t=N_{k-1}+jT_k}^{T-1}
  c(\pi_t,Q_t)\biggr]  & \le &  \hat{C}
\frac{T_k}{T} \le \hat{C}
\frac{T_k}{T_{k-1}'}  \nonumber \\ 
&\le &  \frac{\hat{C}}{k-1} \to 0   \label{eq_Tbound3}
\end{eqnarray}
as $k\to \infty$, where the last inequality
holds since by \eqref{eq_nk} we have $T_k'=n_kT_k \ge k T_{k+1}$ for all $k$. 

Combining  \eqref{eq_JTbound}--\eqref{eq_Tbound3} we obtain
\[
\limsup_{T\to \infty}     E^{\Pi}_{\pi^*}\biggl[\frac{1}{T}\sum_{t=0}^{T-1}
  c(\pi_t,Q_t)\biggr]  \le \limsup_{k\to \infty}  J_{\pi^*}(\Pi_k,T_k) \le J_{\pi^*}
\]
which proves the optimality of $\Pi$.
\qed

\subsection{Proof of  Proposition
  \ref{weakLimitTheorem}} \label{StationaryOptimalProof} \emph{Proof of (a).}
Here we show that any weak limit of $\{v_t\}$ must belong to $\G$.
 For $v\in \P(\P(\R^d)\times
\Q_c)$ and $g\in \C_b(\P(\R^d)\times \Q_c )$ or $f\in \C_b(\P(\R^d))$   define
\[
\langle v, g \rangle \coloneqq \int
g(\pi,Q) v(d\pi\, dQ), \qquad
\langle v, f \rangle \coloneqq \int
f(\pi) v(d\pi\, dQ).
\]
Also define $vP\in \P(\P(\R^d))$ by
\[
vP(A) \coloneqq \int P(\pi_{t+1} \in A | \pi_t=\pi,Q_t=Q)
v(d\pi \, dQ)
\]
for any measurable $A\subset \P(\R^d)$.  Note that  $v\in \G$ is equivalent to
\begin{equation}
  \label{eq_weaklimit}
\langle vP, f \rangle   =\langle v, f\rangle \quad \text{for all $f\in
  \C_b(\P(\R^d))$}.
\end{equation}

From the definition of $v_tP$, we have for any $f\in \C_b(\P(\R^d))$,
\begin{eqnarray}
 \langle v_t, f\rangle - \langle v_t P,  f\rangle & = &
\frac{1}{t} E_{\pi_0}
\biggl[ \, \sum_{i=0}^{t-1} f(\pi_i) -  \sum_{i=1}^t f(\pi_i) \biggr] \nonumber \\*
&= &  \frac{1}{t} E_{\pi_0}\bigl[ \, f(\pi_0) -  f(\pi_t) \bigr]  \to 0  \label{eq_gvconv}
\end{eqnarray}
as $t\to \infty$. Now suppose that $v_{t_k} \to \bar{v}$ weakly along a
subsequence of 
$\{v_t\}$. Then $\langle v_{t_k}, f\rangle \to \langle \bar{v},f\rangle $ for all
$f\in \C_b(\P(\R^d))$, and \eqref{eq_gvconv} implies
\begin{equation}
  \label{eq_vtlim1}
\langle v_{t_k}P, f\rangle \to \langle \bar{v},f\rangle.
\end{equation}

The following lemma is proved  at the end of this section.

\begin{lemma}\label{weakContKernel}
The transition kernel  $ P(d\pi_{t+1} | \pi_t,Q_t)$ is continuous in the
weak-Feller sense, i.e., for any $f\in \C_b(\P(\R^d))$,
\[
Pf(\pi,Q) \coloneqq \int_{\P(\R^d)\times \Q_c} f(\pi')P( d\pi' | \pi,Q)
\]
is continuous on $\S \times \Q_c$,
\end{lemma}

The lemma implies that $Pf\in \C_b(\S\times \Q_c)$, so $\langle v_{t_k},
Pf\rangle \to \langle \bar{v},Pf\rangle $.  However, since for all $v$,
\[
\langle vP, f \rangle = \int_{\P(\R^d)\times \Q_c} f(\pi)P(d\pi'|\pi,Q) v(d\pi\,dQ) = \langle
v, Pf \rangle,
\]
this is equivalent to  $\langle v_{t_k}P,f\rangle \to \langle \bar{v}P,f\rangle
$. Combining this with  \eqref{eq_vtlim1} yields $\langle \bar{v}P,f\rangle =
\langle \bar{v},f\rangle $ which finishes the proof that $\bar{v}\in \G$.

Although $c(\pi,Q)$ is continuous on $\S\times \Q_c$ by
Lemma~\ref{continuityofc}, the limit relation \eqref{eq_cvtconv} does not follow
immediately since $\pi_0$ may not be in $\S$ and thus $v_t$ may not be supported
on $\S\times \Q_c$. However, since $\pi_t\in \S$ for all $t\ge 1$ with
probability~1, we have $v_t(\S\times \Q_c)\ge 1-1/t$, and we can proceed
as follows: Recall that $\S \times \Q_c$ is a closed subset of
$\P(\R^d)\times \Q_c$ by Lemma~\ref{lemma_sqconv} and the topology on
$\P(\R^d)\times \Q_c$ is metrizable. Thus by the Tietze-Urysohn extension
theorem \cite{Dud02} there exists $\tilde{c}\in \C_b(\P(\R^d)\times \Q_c)$ which
coincides with $c$ on $\S\times \Q_c$. Then since $v_{t_n}(\S\times \Q_c)\ge
1-1/t_n$ and both $c$ and $\tilde{c}$
are bounded,
\[
\lim_{n\to \infty}   \int_{\P(\R^d)\times \Q_c} \bigl|
\tilde{c}(\pi,Q)-c(\pi,Q)\bigr| v_{t_n}(d\pi
\, dQ)  =0.
\]
On the other hand,  $v_{t_n}\to \bar{v}$ implies
\begin{eqnarray*}
\lefteqn{  \lim_{n\to  \infty} \int_{\P(\R^d) \times \Q_c }\tilde{c}(\pi,Q)
  v_{t_n}(d\pi\, dQ) } \qquad  \\
&=&
\int_{\P(\R^d) \times \Q_c }\tilde{c}(\pi,Q)
\bar{v}(d \pi\,  dQ)\\
&=& \int_{\P(\R^d) \times \Q_c }c(\pi,Q)
\bar{v}(d \pi\,  dQ),
\end{eqnarray*}
where the last equality holds since $\bar{v}\in \G$ is
supported on $\S\times \Q_c$.
This proves  \eqref{eq_cvtconv}.

\emph{Proof of (b)}. We need the following simple lemma.

\begin{lemma}
\label{lem_relcomp}
Let $H$ be a collection of probability measures on $\P(\R^d)\times \Q_c$ such
that
\[
R\coloneqq \sup_{v\in H} \int_{P(\R^d)\times \Q_c} \biggl(\int_{\R^d}
\|x\|^2\pi(dx) \biggl) v(d\pi\, dQ) <\infty
\]
Then $H$ is tight and is thus relatively compact.
\end{lemma}

\begin{proof}
For any $\alpha>0$ let
\[
K_{\alpha} \coloneqq \biggl\{ \pi\in \P(\R^d): \int_{\R^d} \|x\|^2\pi(dx) \le\alpha \biggr\}.
\]
Then $\pi(\{x: \|x\|^2> L\}) \le \alpha/L$ for all $\pi \in K_{\alpha}$ by
Markov's inequality. Hence $K_{\alpha}$ is tight and thus relatively compact. A standard truncation argument shows that if $\pi_k\to \pi$ (weakly)
for a sequence $\{\pi_k\}$ in $K_{\alpha}$, then
\[
\alpha \ge \limsup_{k\to \infty} \int_{\R^d} \|x\|^2\pi_k(dx) \ge  \int_{\R^d}
\|x\|^2\pi(dx)
\]
so $K_{\alpha}$ is also  closed. Thus $K_{\alpha}$ is compact.

Let $f(\pi)\coloneqq \int_{\R^d} \|x\|^2 \pi(dx)$. Then
\[
\int_{P(\R^d)\times \Q_c} f(\pi) v(d\pi\, dQ) \le R \quad \text{for all $v\in
  H$}
\]
Again by Markov's inequality,
\[
\int_{\P(\R^d)\times Q_c} f(\pi) v(d\pi\, dQ) \ge \alpha v\bigl((K_{\alpha})^c \times \Q_c\bigr)
\]
implying,  for all $v\in H$,
\[
 v(K_{\alpha} \times \Q_c) \ge 1-\frac{R}{\alpha}.
\]
Since $\Q_c$ is compact and $K_{\alpha}$ is compact for all $\alpha>0$, we
obtain that $H$ is tight.
\end{proof}

Let $\Pi$ be an arbitrary fixed policy in $\bar{\Pi}^C_W$, fix the
initial distribution $\delta_{x_0}$, and consider the corresponding sequence of
expected occupation measures $\{v_t\}$. Then
\begin{eqnarray*}
  \lefteqn{  \int_{\P(\R^d)\times \Q_c}  \biggl( \int_{\R^d}  \|x\|^2\pi(dx)
    \biggr) v_t(d\pi\, dQ)  } \\ 
&=& E_{\delta_{x_0}} \biggl[\frac{1}{ t} \sum_{k=0}^{t-1} \|x_k\|^2\biggr] \to
\int_{\R^d}  \|x\|^2 \pi^*(dx) <\infty
\end{eqnarray*}
by Assumption \ref{AssumptionB}. Hence
\begin{equation}
  \label{eq_vttight}
\sup_{t\ge 0}  \int_{\P(\R^d)\times \Q_c}  \biggl( \int_{\R^d}  \|x\|^2\pi(dx)
\biggr) v_t(d\pi\, dQ) <\infty.
\end{equation}
Thus $\{v_t\}$ is relatively compact by  Lemma~\ref{lem_relcomp}, proving part
(b) of the proposition.

\emph{Proof of (c)} We will show that $\G$  is
closed and relatively compact. To show closedness, let $\{v_n\}$ be a sequence
in $\G$ such that $v_n\to \bar{v}$. Using the notation introduced in the proof
of part~(a), we have for any $f\in C_b(\P(\R^d))$ by \eqref{eq_weaklimit},
\[
 \langle v_nP, f\rangle =  \langle v_n,f\rangle \to  \langle \bar{v} ,f\rangle.
\]
But we also have
\[
 \langle v_nP, f\rangle=  \langle v_n, P f\rangle \to   \langle \bar{v}, P
 f\rangle  =  \langle \bar{v}P, f\rangle ,
\]
where the limit holds by the weak-Feller property of $P$
(Lemma~\ref{weakContKernel}).  Thus $  \langle \bar{v}P , f\rangle =   \langle
\bar{v}, f\rangle$, showing that $\bar{v}\in \G$. Hence $\G$ is closed.

To show relative compactness, recall from \eqref{updateQuantizationNext} the
conditional distributions
\[
\hat{\pi}(m,\pi,Q)(dx_{t+1})= P(dx_{t+1}| \pi_t=\pi,Q_t=Q, q_t = m)
\]
for $m=1,\ldots,M$. 
For any $(\pi,Q)$ and Borel set $A\subset \R^d$,
\begin{eqnarray}
\lefteqn{  \int_{\P(\R^d)} \pi'(A) P(d\pi'| \pi, Q) } \qquad  \nonumber  \\*
&&  \hspace{-35pt} = \sum_{m=1}^M \hat{\pi}(m,\pi,Q)(A)  \,P\bigl(
\hat{\pi}(m,\pi,Q)| \pi, 
  Q\bigr)   \nonumber   \\
&&  \hspace{-35pt}   =\sum_{m=1}^M   \biggl(  \frac{1}{\pi(Q^{-1}(m))}
\!\!\!\!  \!\!\! \int\limits_{Q^{-1}(m)} \!\!\!  \!\!\! 
P(x_{t+1}\!\in\! A| x_t) \pi(dx_t) \biggr) \pi(Q^{-1}(m))   \nonumber  \\
&&  \hspace{-35pt}  =   \int_{\R^d}
P(x_{t+1}\in A| x_t) \pi(dx_t). \label{eq_inv1}
\end{eqnarray}
Now let $v\in \G$ and consider the ``average'' $\pi_v$ under $v$ determined by
\[
\pi_v(A) = \int_{\P(\R^d)\times \Q_c} \pi(A) v(d\pi\, dQ) = \int_{\P(\R^d)}
\pi(A)\hat{v}(d\pi),
\]
where $\hat{v}$ is obtained from $v(d\pi\, dQ)=\bar{\eta}(dQ|\pi) \hat{v}(d\pi)$.
Recall that $v$ is supported on $\S\times \Q_c$.
If $A$  has boundary of zero Lebesgue measure, the mapping $\pi\mapsto \pi(A)$ is
continuous on $\S$ and
the definition of $\G$ implies
\begin{eqnarray}
\lefteqn{ \pi_v(A)} \nonumber  \\
 & &  \hspace{-10pt} =\int_{\P(\R^d)\times \Q_c}  \pi(A) v(d\pi\, dQ) \nonumber \\
& &  \hspace{-10pt}  = \int_{\P(\R^d)\times \Q_c}
\int_{\P(\R^d)}  \pi'(A) P( d\pi'|
\pi, Q) v(d\pi\, dQ)  \nonumber \\
&  & \hspace{-15pt}   =  \int\limits_{\P(\R^d)}\int\limits_{\Q_c}
\int\limits_{\P(\R^d)}  \!\!\!\! \pi'(A) P( d\pi'|
\pi, Q)\bar{\eta}(dQ|\pi)  \hat{v}(d\pi).   \label{eq_vinariant}
\end{eqnarray}
Substituting \eqref{eq_inv1} into the last integral, we obtain
\begin{eqnarray*}
\pi_v(A)\!\!\! & =&  \!\!\!  \int\limits_{\P(\R^d)}\int\limits_{\Q_c}
  \int\limits_{\R^d}
P(x_{t+1}\in A| x_t) \pi(dx_t) \bar{\eta}(dQ|\pi)  \hat{v}(d\pi)   \\
&=&     \int_{\R^d}
P(x_{t+1}\in A| x_t) \pi_v(dx_t).
\end{eqnarray*}
Since the  Borel sets in $\R^d$ having boundaries of zero Lebesgue
measure form  a  separating class for $\P(\R^d)$, the above holds for all Borel
sets $A$, implying that $\pi_v=\pi^*$, the unique invariant measure for
$\{x_t\}$. Thus
\begin{eqnarray}
\lefteqn{ \int_{P(\R^d)\times \Q_c} \biggl(\int_{\R^d}
\|x\|^2\pi(dx) \biggl) v(d\pi\, dQ) } \nonumber \\
& =& \int_{\R^d} \|x\|^2\pi_v(dx) =
\int_{\R^d} \|x\|^2\pi^*(dx)   \label{eq_moments} 
\end{eqnarray}
for all $v\in \G$. Since the last integral is finite by
Assumption~\ref{AssumptionB}, Lemma~\ref{lem_relcomp} implies that $\G$ is
relatively compact.  \qed

\emph{Proof of Lemma~\ref{weakContKernel}.}
Consider a sequence $\{(\pi_n,Q_n)\}$ converging to some
$(\pi,Q)$ in $\S\times \Q_c$. Then for any $f\in \C_b(\P(\R^d))$,
\begin{eqnarray*}
\lefteqn{  \int\limits_{\P(\R^d)\times \Q_c} \!\!\! f(\pi')P( d\pi' | \pi_n,Q_n)
-  \!\!\!\! \int\limits_{\P(\R^d)\times \Q_c}  \!\!\! f(\pi')P( d\pi' | \pi,Q) }
\nonumber \\ 
&=& \sum_{m=1}^M \Bigl( f(\hat{\pi}(m,\pi_n,Q_n))P(\hat{\pi}(m,\pi_n,Q_n)|
\pi_n,Q_n) \\ 
& &  \qquad \qquad \mbox{}  - 
 f(\hat{\pi}(m,\pi,Q))P(\hat{\pi}(m,\pi,Q)| \pi,Q) \Bigr) \nonumber \\
&=& \sum_{m=1}^M \Bigl( f(\hat{\pi}(m,\pi_n,Q_n)) \pi_n\bigl(Q_n^{-1}(m)\bigr)
\\
& & \qquad \qquad \mbox{} -
 f(\hat{\pi}(m,\pi,Q)) \pi\bigl(Q^{-1}(m)\bigr) \Bigr). 
\end{eqnarray*}
From Lemma~\ref{lemma_sqconv} we have that $\pi_nQ_n\to \pi Q$ in total
variation which implies via Lemma~\ref{TotalVarConvBins} that
$\hat{\pi}(m,\pi_n,Q_n) \to \hat{\pi}(m,\pi,Q)$ in total variation and thus
weakly for all $m$ with $\pi\bigl(Q^{-1}(m)\bigr)>0$. The proof of the same
lemma shows that $\pi_n\bigl(Q_n^{-1}(m)\bigr) \to \pi\bigl(Q^{-1}(m)\bigr)$ for
all $m=1,\ldots M$. Since $f$ is continuous and bounded, the last sum converges
to zero as $n\to \infty$, proving the claim of the lemma. \qed

\section{Acknowledgements}

We are grateful to Vivek S. Borkar and Naci Saldi for technical discussions
related to the paper. We also  thank two anonymous reviewers for 
constructive comments. 

\pagebreak[2]


\begin{thebibliography}{bib}
\bibitem{AbayaWise} E. A. Abaya and G. L. Wise, ``Convergence of
   vector quantizers with applications to optimal quantization,'' {\em
     SIAM Journal on Applied Mathematics}, vol. 44, pp.  183--189, 1984.

\bibitem{Ant12} A.~Antos, ``On codecell convexity of optimal multiresolution
scalar quantizers for continuous sources,'' \emph{IEEE Trans. Inform. Theory},
vol.~58, no.~2, pp.\ 1147--1157, Feb.\ 2012

\bibitem{survey} A. Arapostathis, V. S. Borkar, E. Fernandez-Gaucherand, M. K. Ghosh and S. I. Marcus,
``Discrete-Time controlled Markov processes with average cost criterion: A survey," {\em SIAM J. Control and Optimization}, vol. 31, pp. 282--344, 1993.

\bibitem{ABG} A. Arapostathis, V. S. Borkar and M. K. Ghosh, {\em Ergodic Control of Diffusion Processes}, Cambridge University Press, 2012.

 \bibitem{AsWe13} H.~Asnani and T.~Weissman, ``On real time coding with
   limited lookahead,''  \emph{IEEE Trans. Inform.
  Theory}, vol. 59, no. 6, pp.~3582--3606, Jun. 2013. 


\bibitem{BaSkJo11}  L. Bao, M. Skoglund, and K. H. Johansson,  ``Iterative
  encoder-controller design for feedback control over noisy channels,'' \emph{IEEE
  Trans. on Automatic Control}, vol. 57, no. 2, pp. 265--278, Feb.  2011. 


\bibitem{Bil86}
P.~Billingsley, {\em Probability and Measure}. New York: Wiley, 2nd~ed., 1986.



\bibitem{Borkar} V. S. Borkar, ``Convex analytic methods in Markov Decision
  Processes," {\em Handbook of Markov Decision Processes: Methods and
    Applications}, Kluwer, Boston, 2002.

\bibitem{BorkarLQG} V. S. Borkar and S. K. Mitter, ``LQG control with communication constraints," in {\em Kailath
Festschrift}, Kluwer, Boston, 1997.

\bibitem{BorkarMitterSahaiTatikonda} V.S. Borkar, S.K. Mitter, A. Sahai and S. Tatikonda, ``Sequential source coding: An optimization viewpoint,''{\em Proc. IEEE Conference on Decision and Control}, pp. 1035--1042, Seville, Spain, Dec. 2005.

\bibitem{BorkarMitterTatikonda} V. S. Borkar, S. K. Mitter, and S. Tatikonda,
  ``Optimal sequential vector quantization of Markov sources,'' {\em SIAM    J. Control and Optimization}, vol. 40, no. 1, pp. 135--148, 2001.

\bibitem{Dud02} R. M. Dudley, \emph{Real Analysis and Probability},
Cambridge University Press, Cambridge, 2nd ed., 2002.


\bibitem{DuWu09} S.~Dumitrescu and X. Wu, ``On properties of locally optimal
  multiple description scalar quantizers with convex cells,'' \emph{IEEE
    Trans. Inform. Theory}, vol.~55, no.~12, pp.\ 5591--5606, Dec.\ 2009.



\bibitem{Gikhman} I. I. Gikhman and A. V. Skorokhod, {\em Controlled Stochastic
    Processes}, Springer, New York, 1979.



\bibitem{GrNe98}
R.~M. Gray and D.~L. Neuhoff, ``Quantization,'' {\em IEEE Trans. Inform.
  Theory, {\rm (Special Commemorative Issue)}}, vol.~44, pp.~2325--2383, Oct.
  1998.

\bibitem{GyorgyLinder2003} A. Gy\"orgy and T. Linder, ``Codecell convexity in
  optimal entropy-constrained vector quantization,'' {\em IEEE Trans. Inf. Theory}, vol. 49, pp. 1821--1828, July 2003.

\bibitem{HernandezLermaMCP} O. Hernandez-Lerma, J.B. Lasserre, {\em
    Discrete-Time Markov Control Processes, Basic Optimality Criteria},
  Springer, New York, 1996.

\bibitem{HernandezLermaMCP1} O. Hernandez-Lerma, J.B. Lasserre, {\em
    Further topics on discrete-time Markov control processes},
  Springer, New York, 1999.


\bibitem{HernandezLermaLasserre} O. Hernandez-Lerma, J.B. Lasserre, {\em Markov Chains and Invariant Probabilities}, Birkh\"auser, Basel, 2003.

\bibitem{JaGo13} T.~Javidi and A. Goldsmith, ``Dynamic joint source-channel
  coding with feedback ,'' \emph{Proc. IEEE International Symposium on Information
    Theory}, pp.\ 16-20, Istanbul, 2013.

\bibitem{KaMe12} Y. Kaspi and N. Merhav,  ``Structure theorems for real-time
  variable rate coding with and without side information,'' \emph{IEEE
    Trans. Inf. Theory}, vol. 58, no.12, pp. 7135--7153, Dec. 2012.

\bibitem{LiZa06} T.~Linder and R.~Zamir, ``Causal coding of stationary sources
and individual sequences with high resolution,'' \emph{IEEE Trans.\
Inform.\ Theory}, vol.\ 52, no.\ 2, pp.\ 662--680, Feb.\ 2006. 

\bibitem{MahajanTeneketzisJSAC} A. Mahajan and D. Teneketzis,  ``On the design of globally optimal communication strategies for real-time noisy communication with noisy feedback", {\em IEEE Journal on Special Areas in Communications}, vol. 28, pp. 580--595, May 2008.

\bibitem{Mahajan09} A. Mahajan and D. Teneketzis,  ``Optimal design of sequential real-time communication systems", {\em IEEE Transactions on Inform. Theory},  vol. 55, pp. 5317--5338, November 2009.

\bibitem{CTCN} S. P. Meyn, {\em Control Techniques for Complex Networks}, Cambridge, UK: Cambridge University Press.


\bibitem{MeynBook} S. P. Meyn and R. Tweedie, {\em Markov Chains and Stochastic
Stability}, Springer, London, 1993.

\bibitem{MuEf08} D.~Muresan and M. Effros ``Quantization as histogram
  segmentation: Optimal scalar quantizer design in network systems,'' \emph{IEEE
    Trans. Inform. Theory}, vol.~54, no.~1, pp.\ 344--366, Jan.\ 2008.


\bibitem{NaFaZaEv07}  G. N. Nair, F. Fagnani, S. Zampieri, and J. R. Evans,
  ``Feedback control under data constraints: an overview,'' \emph{Proceedings of
    the IEEE}, vol. 95, no. 1, pp. 108--137, Jan. 2007.

\bibitem{NeGi82}
D.~L. Neuhoff and R.~K. Gilbert, ``Causal source codes,'' {\em IEEE Trans.
  Inform. Theory}, vol.~28, pp.~701--713, Sep. 1982.

\bibitem{Pollard} D. Pollard, ``Quantization and the method of
   $k$-means,'' {\em IEEE Trans. Inf. Theory}, vol. 28, pp. 199--205, 1982.

 \bibitem{SLY} N. Saldi, T. Linder and S. Y\"uksel, ``Randomized quantization
   and optimal design with a marginal constraint,'' {\em Proc. IEEE International
     Symposium on Information Theory}, pp.\ 2349--2353, Istanbul, 2013.

\bibitem{TatikondaMitter} S. Tatikonda and S. Mitter, ``The capacity of channels
  with feedback,'' {\em IEEE Trans. Inf. Theory}, vol.\ 55, pp.\ 323--349, Jan. 2009.

\bibitem{TatikondaSahaiMitter} S. Tatikonda, A. Sahai, and S. Mitter,
  ``Stochastic linear control over a communication channels'', {\em IEEE
    Trans. Aut. Control}, vol. 49, pp. 1549--1561, Sept. 2004.

\bibitem{Teneketzis} D. Teneketzis, ``On the structure of optimal real-time
  encoders and decoders in noisy communication,'' {\em IEEE Trans. Inf. Theory},
  vol. 52, pp. 4017--4035, Sep. 2006.

\bibitem{WalrandVaraiya} J. C. Walrand and P. Varaiya, ``Optimal causal
coding-decoding problems,'' {\em IEEE Trans. Inf. Theory}, vol. 19, pp. 814--820, Nov. 1983.

\bibitem{WeMe05}  T.~Weissman and N.~Merhav, ``On causal source codes with side information,'' \emph{IEEE Trans. Inform. Theory}, vol.~51,
 no.~11,  pp.~4003--4013, Nov. 2005.



\bibitem{Witsenhausen} H. S. Witsenhausen, ``On the structure of  real-time  source coders,'' {\em Bell Syst. Tech. J.}, vol. 58, pp. 1437--1451, Jul./Aug. 1979.

\bibitem{YukITArXivMultiTerminal} S. Y\"uksel, ``On optimal causal coding of
  partially observed Markov sources in single and multi-terminal settings," {\em
    IEEE Trans. Inf. Theory}, vol. 59, pp. 424--437, Jan. 2013.

\bibitem{YukselAllerton2012} S. Y\"uksel, ``Jointly optimal LQG quantization and
  control policies for multi-dimensional systems'' \emph{IEEE Trans.\ on
    Automatic Control}, vol. 59, Jun.  2014.

\bibitem{YukselBasarBook} S. Y\"uksel, T. Ba\c{s}ar, {\em Stochastic Networked Control Systems: Stabilization and Optimization under Information Constraints}, Birkh\"auser, Boston, MA, 2013.

\bibitem{YukLinSIAM2010} S. Y\"uksel and T. Linder, ``Optimization and convergence of observation channels in stochastic control,'' {\em SIAM Journal on Control and Optimization}, vol. 50, no. 2, pp. 864--887, 2012.

\bibitem{YukMeynTAC2013}  S. Y\"uksel and S. P. Meyn, ''Random-time,
state-dependent stochastic drift for Markov chains and application to
stochastic stabilization over erasure channels," \emph{IEEE Trans.\ Automatic
Control}, vol.~58, pp.\ 47--59, Jan.\ 2013.


\end{thebibliography}
\end{document}